\documentclass{amsart}
\usepackage{all2021,xypic}
\renewcommand{\Vec}{\mathbf{Vec}}
\newcommand{\EE}{\mathbb{E}}
\newcommand{\Sh}{\mathrm{Sh}}
\newcommand{\omin}{\ominus}
\newcommand{\comment}[1]{}
\newcommand{\FI}{\mathbf{FI}}
\newcommand{\ab}{c} %for the notation clash in the proof of the embedding theorem

\begin{document}

\title{A tensor restriction theorem over finite fields}

\author{Andreas Blatter} 
\address{Mathematical Institute, University of Bern,
Alpeneggstrasse 22, 3012 Bern, Switzerland}
\email{andreas.blatter@unibe.ch}

\author{Jan Draisma} 
\address{Mathematial Institute, University of Bern, Sidlerstrasse 5, 3012
Bern, Switzerland; and Department of Mathematics and Computer Science,
P.O. Box 513, 5600 MB, Eindhoven, the Netherlands}
\email{jan.draisma@unibe.ch}

\author{Filip Rupniewski}
\address{Mathematical Institute, University of Bern,
Alpeneggstrasse 22, 3012 Bern, Switzerland}
\email{filip.rupniewski@unibe.ch}

\thanks{AB and FR are supported by Swiss National Science Foundation
(SNSF) project grant 200021\_191981, and JD is partially supported by
that grant and partially supported by Vici grant 639.033.514 from the
Netherlands Organisation for Scientific Research (NWO). JD thanks the
Institute for Advanced Study for the excellent working conditions,
under which part of this project was carried out.}

\maketitle

\begin{abstract}
Restriction is a natural quasi-order on $d$-way tensors. We establish
a remarkable aspect of this quasi-order in the case of tensors over a
fixed finite field---namely, that it is a well-quasi-order: it admits no
infinite antichains and no infinite strictly decreasing sequences. This
result, reminiscent of the graph minor theorem, has important consequences
for an arbitrary restriction-closed tensor property $X$. For instance,
$X$ admits a characterisation by finitely many forbidden restrictions
and can be tested by looking at subtensors of a fixed size.  Our
proof involves an induction over polynomial generic representations,
establishes a generalisation of the tensor restriction theorem to other
such representations (e.g.~homogeneous polynomials of a fixed degree),
and also describes the coarse structure of any 
restriction-closed property. 
\end{abstract}

\section{Introduction and results}

\subsection{Tensor restriction}

Let $K$ be a finite field and let $d$ be a natural number. This paper
concerns properties of $d$-way tensors that are preserved under taking
linear maps. For a vector space $V$ over $K$ we denote by $V^{\otimes d}$ 
the $d$-fold tensor product $V \otimes V \otimes \cdots \otimes V$ over
$K$, and for a linear map $\phi:V \to W$ we denote by $\phi^{\otimes
d}:V^{\otimes d} \to W^{\otimes d}$ the linear map determined by
\[ \phi^{\otimes d}(v_1 \otimes \cdots \otimes
v_d):=\phi(v_1) \otimes \cdots \otimes \phi(v_d). \]

\begin{de}
Let $V,W$ be finite-dimensional vector spaces over $K$ and let $S
\in V^{\otimes d}$ and $T \in W^{\otimes d}$. We call $T$ a {\em
restriction} of $S$ if there exists a linear map $\phi:V \to W$ such that
$\phi^{\otimes d} S = T$. We then write $S \succeq T$.
\end{de}

The rationale for this terminology is that $S$ can be thought of
as a multilinear map $(V^*)^d \to K$, and composing this 
map with $(\phi^*)^d:(W^*)^d \to (V^*)^d$ gives the multilinear map
$T$. In particular, if $\phi^*$ is injective, so that we can use it to
identify $W^*$ with a subspace of $V^*$, then we can think of $T$ as
the restriction of $S$ to the subspace $(W^*)^d$. 

\begin{re} \label{re:Tensors}
Much literature on tensors considers tensor products $V_1 \otimes \cdots
\otimes V_d$ of different vector spaces $V_i$, and for restriction
allows the application of distinct linear maps $\phi_i: V_i \to W_i$
to the individual factors. The theorems that we will prove imply the
corresponding theorems for this setting; see
Remark~\ref{re:Multivariate}.
\end{re}

\subsection{The tensor restriction theorem over finite fields}

The relation $\succeq$ is reflexive and transitive, so it
is a quasi-order on tensors over $K$. We will prove that 
this quasi-order is a well-quasi-order.

\begin{thm}[Tensor restriction theorem]
\label{thm:TensorRestriction}
Fix a natural number $d$. For every $i \in \NN$ let $V_i$ be a
finite-dimensional vector space over the finite field $K$ and let $T_i
\in V_i^{\otimes d}$. Then there exist $i<j$ such that $T_j \succeq T_i$.
\end{thm}

As the following example shows, the requirement that $K$ be
finite is essential. 

\begin{ex}
If $|K|=\infty$, then the statement of the theorem fails
already for $d=2$. Indeed, if $\cha K \neq 2$, then consider
the matrices
\[ M_a:=\begin{bmatrix} 1 & a \\ -a & 1 \end{bmatrix} \in
(K^2)^{\otimes 2}  \]
for $a$ ranging through $K$. If $M_a \succeq M_b$, then there exists
a matrix $g \in \GL_2(K)$ such that $g M_a g^T = M_b$. Looking at the
symmetric parts of $M_a$ and $M_b$, we find that $g I g^T=I$, so $g$
is an orthogonal matrix and $M_a,M_b$ have the same characteristic
polynomial. But the characteristic polynomial of $M_a$ equals $(t-1)^2
- a^2$, so $M_a \succeq M_b$ holds (if and) only if $a^2=b^2$. Since
$|K|=\infty$, we have found infinitely many $2$-way tensors that are
incomparable with respect to $\succeq$. A similar construction works
when $\cha K=2$. It is easy to see that the failure for $d=2$ implies
the failure for all larger $d$.
\end{ex}

\subsection{Consequences of the tensor restriction theorem}

The tensor restriction theorem is reminiscent of the celebrated graph
minor theorem \cite{Robertson04}, which says that finite graphs are
well-quasi-ordered by the minor order. We are not aware of any logical
dependence between these theorems, but the tensor restriction theorem
has similar far-reaching consequences for tensors as the graph minor
theorem has for graphs. These consequences are best formulated using
the following notion.

\begin{de}
A {\em restriction-closed property} of $d$-way tensors is a property such
that if a tensor $S$ has it, and $S \succeq T$ holds, then also $T$ has it. We
can identify such a property with the data of a subset $X(V) \subseteq
V^{\otimes d}$ for all every finite-dimensional vector space $V$ over $K$,
such that if $\phi:V \to W$ is a linear map, then $\phi^{\otimes d}$
maps $X(V)$ into $X(W)$. 
\end{de}

\begin{ex}
Let $T \in W^{\otimes d}$. Then the property of {\em not} having $T$ as
a restriction is restriction-closed. We denote this property by $X_{\not
\succeq T}$. 
\end{ex}

If $X$ is a restriction-closed property, and $T$ is a tensor
that does not satisfy it, then we call $T$ a {\em forbidden
restriction} for $X$. 

\begin{cor} \label{cor:Main}
For $d$-way tensors over the fixed finite field $K$ the
following hold.
\begin{enumerate}
\item Restriction-closed properties satisfy the descending
chain condition: any chain $X_1 \supseteq X_2 \supseteq
\ldots$ of such properties stabilises. 

\item Every restriction-closed property $X$ is characterised by finitely
many forbidden restrictions, i.e., we have $X=\bigcap_{i=1}^k X_{\not
\succeq T_i}$ for some $k$ and some tensors $T_i \in V_i^{\otimes d}$.

\item For every restriction-closed property $X$ there exists a
finite-dimensional vector space $U$ such that for any $V$ and any $T \in
V^{\otimes d}$, we have $T \in X(V)$ if (and only if) $\phi^{\otimes d}
T \in X(U)$ for all linear maps $\phi:V \to U$.

\item For every restriction-closed property $X$ there exists
a number $n_0$ such that a tensor $T \in (K^n)^{\otimes d}$
satisfies $X$ if and only if for every subset $S \subseteq
[n]:=\{1,\ldots,n\}$ of size $n_0$ the sub-tensor of $T$ in $(K^S)^{\otimes
d}$ satisfies $X$.

\item For every restriction-closed property $X$ there exists
a polynomial-time deterministic algorithm that on input $n$ and a $T
\in (K^n)^{\otimes d}$ decides whether $T$ satisfies $X$.
\end{enumerate}
\end{cor}

The proofs of (1), (2), and (3) are straightforward from the
tensor restriction theorem, and conversely the tensor
restriction theorem follows from each of these. 

\begin{proof}[Proofs of (1),(2),(3) from the tensor
restriction theorem and vice versa.]

Assuming the tensor restriction theorem, we prove (1): whenever $X_i$
and $X_{i+1}$ are not the same property, there exists a tensor $T_i$
that satisfies $X_i$ but not $X_{i+1}$. For $i<j$ we then have $T_j \not
\succeq T_i$, and hence $X_i \neq X_{i+1}$ holds only finitely many times.

Next we prove (1) $\Rightarrow$ (2). Start with $k=0$. While $X$ is
strictly contained in $X_k:=\bigcap_{i=1}^k X_{\not \succeq T_i}$, there
is a tensor $T_{k+1}$ that does not satisfy $X$ but does not have any of
the tensors $T_1,\ldots,T_k$ as a restriction. This yields a strictly
descending chain $X_0 \supsetneq X_1 \supsetneq \ldots$, which by (1)
must terminate, so that $X$ is equal to $X_k$ for some $k$.

For (2) $\Rightarrow$ (3), we take for $U$ any space of dimension at
least that of all of the spaces $V_i,\ i=1,\ldots,k$, where $T_i \in
V_i^{\otimes d}$. If $T \in V^{\otimes d}$ does not lie in
$X(V)$, then it has a restriction equal to some
$T_i$, so that $\psi^{\otimes d} T=T_i$ for some linear map $\psi:V \to
V_i$. Now $\psi$ factors
via a linear map $\phi:V \to U$, and it follows that $\phi^{\otimes d}
T \not \in X(U)$.

Finally, (3) implies the tensor restriction theorem: let $T_i \in
V_i^{\otimes d},\ i=1,2,\ldots$, and define $X_n:=\bigcap_{i=1}^n
X_{\not \succeq T_i}$ and $X:=\bigcap_{i=1}^\infty X_{\not \succeq T_i}$.
Let $U$ be as in (3) for $X$. Then, since $X_0(U)$ is a finite set, the
chain of subsets
\[ X_0(U) \supseteq X_1(U) \supseteq \ldots \]
stabilises after finitely many steps: $X_n(U)=X(U)$. Then in particular
$T_{n+1}$, which is not in $X(V_{n+1})$, is not in $X_n(V_{n+1})$,
which means that it must have some $T_i$ with $i \leq n$ as a restriction.
\end{proof}

Note that the difference between (3) and (4) is that in (4) we only
consider coordinate projections $K^n \to K^I$. The proof of (4) is
slightly more involved and deferred to \S\ref{ssec:FI}. Once (4) is established,
however, (5) is immediate, since there are only $\binom{n}{n_0}$ subsets
$I$ of size $n_0$, and this is a polynomial in $n$. Note that for this
proof it does not matter whether the tensor in (5) is given in dense
input form (an array of $n^d$ elements from $K$) or in sparse form (a
list of tuples $(i_1,\ldots,i_d,a)$ where $(i_1,\ldots,i_d)$ specifies
the position of a tensor entry and $a \in K$ its value).

\begin{re} \label{re:InfiniteFields}
Versions of (1),(3),(4), and (5) also hold for restriction-closed tensor
properties over an {\em infinite} field, provided that the tensor property
can be expressed by polynomial equations in the tensor
entries. For (1),(3) this follows from \cite{Draisma17}. For
(4),(5), this follows from (1),(3) and the technique in
\S\ref{ssec:FI} below.
\end{re}

\subsection{Restriction-monotone functions}

Tensor restriction plays an important role in theoretical computer
science, in particular through many notions of tensor rank,
of which we briefly discuss two here. 

\begin{de}
The {\em rank} $\rk(S)$ of $S \in V^{\otimes d}$ is the minimal $r$
such that $S$ can be written as 
\[ S=\sum_{i=1}^r v_{i,1} \otimes \cdots \otimes v_{i,d} \]
for suitable vectors $v_{i,1},\ldots,v_{i,d}$. 
The {\em asymptotic rank} of $S$ is the limit  
\[ \lim_{t \to \infty} \sqrt[t]{\rk(S^{\boxtimes t})}, \]
where the {\em vertical tensor product} $S^{\boxtimes t}$ is the $d$-way
tensor in $(V^{\otimes t})^{\otimes d}$ obtained by tensoring $t$ copies
of $S$ and grouping together the $t$ copies of the first copy of $V$,
the $t$ copies of the second copy, and so on.
\end{de}

If $S \succeq T$, then $\rk(S) \geq \rk(T)$ and $S^{\boxtimes t} \succeq
T^{\boxtimes t}$ for every natural number $t$, so that also the asymptotic
rank of $S$ is at least that of $T$. This shows that rank and asymptotic
rank are both monotone in the following sense.

\begin{de}
A function $f$ that assigns to any $d$-way tensor a real
number is called {\em restriction-monotone} if $S \succeq T$
implies that $f(S) \geq f(T)$. 
\end{de}

\begin{cor} \label{cor:WellOrdered}
Let $f$ be any restriction-monotone function on $d$-way tensors over
the finite field $K$. Then the set of values of $f$ in $\RR$ is a
well-ordered set.
\end{cor}

\begin{proof}
If not, then there exist tensors $T_1,T_2,\ldots$ on which $f$ takes
values $a_1>a_2>\ldots$. Let $X_{\leq a_i}$ be the tensor property of
having $f$-value $\leq a_i$. Since $f$ is restriction-monotone, this property is
restriction-closed. Furthermore, since $T_i \in X_{\leq a_i}
\setminus X_{\leq a_{i+1}}$, we have 
\[ X_{\leq a_1} \supsetneq X_{\leq a_2} \supsetneq \ldots \]
But this contradicts Corollary~\ref{cor:Main}, part (1). 
\end{proof}

In particular, the set of asymptotic ranks of $d$-way tensors over a
fixed finite field is well-ordered. 

\begin{ex}
Take $d=3$. By Corollary~\ref{cor:WellOrdered}, the set $S \subseteq
\RR_{\geq 0}$ of asymptotic ranks of $3$-way tensors is well-ordered. This
means that $S \setminus [0,4]$ contains a minimal element $4+\epsilon$
with $\epsilon>0$. Hence in particular, the asymptotic rank of $2 \times
2$-matrix multiplication, a tensor in $K^4 \otimes K^4 \otimes K^4$,
is either $4$ (which is equivalent to the well-known conjecture that the
exponent of matrix multiplication over $K$ is $2$; see \cite{Conner21})
or at least $4+\epsilon$. We point out, though, that we do not know
whether asymptotic ranks of tensors over an {\em infinite} field are
well-ordered, because the property of having asymptotic rank at most some
real number is not (evidently, at least!) a Zariski-closed property---see
Remark~\ref{re:InfiniteFields}.
\end{ex}

\comment{
\begin{ex}
By \cite{}, any $d$-way tensor either has asymptotic rank
$0$ (in which case it is zero), or $1$ 
(in which case it is a tensor product of vectors), or asymptotic
rank $\geq ..$. The existence of such gaps is a consequence of
Corollary~\ref{cor:WellOrdered}.
\end{ex}
}

\begin{ex}
Another notion of rank that is restriction-monotone is {\em analytic
rank} \cite{Lovett19}. Fix
a non-trivial character (group homomorphism) $\chi:(K,+) \to
(\CC^*,\cdot)$. Thinking of $T \in V^{\otimes d}$ as a multilinear form
$(V^*)^d \to K$, the analytic rank of $T$ equals
\[ -\log_{|K|} \EE(\chi(T(x_1,\ldots,x_d))), \]
where $\EE$ stands for expectation in the probabilistic model where 
$(x_1,\ldots,x_d)$ is picked uniformly at random in $(V^*)^d$.
The analytic rank is
restriction-monotone be Lemma~\ref{lm:Analytic} below. Hence, by
Corollary~\ref{cor:WellOrdered}, the set of analytic ranks of $d$-linear
forms over $K$ is a well-ordered subset of the real numbers.
\end{ex}

The following is well-known to experts, but we did not find a proof in
the published literature, so we provide one here.

\begin{lm} \label{lm:Analytic}
The analytic rank is restriction-monotone.
\end{lm}

\begin{proof}
It is convenient to see this in the more general setting of
Remark~\ref{re:Tensors}, where we have different vector spaces
$V_1,\ldots,V_d$ and $T \in V_1 \otimes \cdots \otimes V_d$ is a regarded
as a multilinear function $V_1^* \times \cdots \times V_d^* \to K$.

Consider a linear map $\phi:V_d \to W$ and define $T':=\id_{V_1} \otimes
\cdots \otimes \id_{V_{d-1}} \otimes \phi$. For fixed
$(x_1,\ldots,x_{d-1}) \in V_1^* \times \cdots \times
V_{d-1}^*$, the linear form $T(x_1,\ldots,x_{d-1},\cdot)$ is
either zero, in which case
$\chi(T(x_1,\ldots,x_{d-1},x_d))=1$ for all $|V_d|$ choices of $x_d$,
or it is nonzero, in which case, as $x_d$ varies through $V$,
$T(x_1,\ldots,x_d)$ takes all values equally often, and
therefore the values of $\chi$ cancel out. We conclude that the
expectation in the analytic rank of $T$ equals 
\[ a |V_d| / (|V_1| \cdots |V_d|) \]
where $a$ is the number of tuples $(x_1,\ldots,x_{d-1})$ for which
the linear form is zero. By the same reasoning, 
the expectation in the analytic rank of $T'$ equals
\[ a' |W|  / (|V_1| \cdots |V_{d-1}| \cdot |W|) \]
where now $a'$ is the number of tuples
$(x_1,\ldots,x_{d-1},\cdot)$ for
which $T(x_1,\ldots,x_{d-1})$ is zero {\em on the image of $\phi^*:W^*
\to V^*$}. Now $a' \geq a$ and therefore 
the expression for $T'$ is at least
that for $T$. Taking $-\log_{|K|}$ on both sides, and
repeating this argument with linear maps in the other $d-1$
tensor factors, we are done. 
\end{proof}

\subsection{Generic representations}

Let $\Vec$ be the category of finite-dimensional vector spaces over the
finite field $K$.

\begin{de}
A {\em generic representation} is a functor $F:\Vec \to \Vec$.  
\end{de}

The terminology is explained by the observation that if $F$ is a generic
representation, then for each $n$, $F(K^n)$ is a representation of the
finite group $\GL_n(K)$ and of the finite monoid $\End(K^n)$ of $n \times
n$-matrices. Generic representations can therefore be thought of as
sequences of representations of $\End(K^n)$, one for each $n$, that
depend in a suitably generic manner on $n$. Generic representations form
an abelian category in which the morphisms are natural transformations.

\begin{ex}
Here are two rather different examples of generic
representations: 
\begin{enumerate}
\item the functor $T^d$ that sends $V$ to $V^{\otimes d}$ and $\phi:V
\to W$ to $\phi^{\otimes d}$; and 
\item the functor that sends $V$ to the $K$-vector space $KV$
with basis $V$ and $\phi:V \to W$ to the unique linear map
$KV \to KW$ that sends the basis vector $v \in V$ to the
basis vector $\phi(v) \in W$. \qedhere
\end{enumerate}
\end{ex}

The following beautiful theorem characterises a particularly
nice class of generic representations. 

\begin{thm}[{\cite[Theorem 4.14]{Kuhn94}}] \label{thm:Kuhn}
For a generic representation $F: \Vec \to \Vec$ the
following properties are equivalent: 
\begin{enumerate}
\item $F$ has a finite composition series in the abelian
category of generic representations;
\item the function $d_F:\ZZ_{\geq 0} \to \ZZ_{\geq 0}$
defined by $d_F(n):=\dim F(K^n)$ is (bounded above by) a
polynomial in $n$; and 
\item $F$ is a subquotient of a finite direct sum $T^{d_1}
\oplus \cdots \oplus T^{d_n}$ for suitable $d_1,\ldots,d_n
\in \ZZ_{\geq 0}$. 
\end{enumerate}
\end{thm}

We call a generic representations satisfying any of the equivalent
properties above {\em polynomial}. Often, we will drop the adjective
generic and just speak of {\em polynomial representations}.

\begin{ex}
The generic representation $V \mapsto V^{\otimes d}$ is
polynomial, and so is the generic representation $V \mapsto
S^d V$. The generic representation $V \mapsto KV$ is not
polynomial, because $\dim KV=|V|=|K|^{\dim V}$ is
exponential in $\dim V$.
\end{ex}

\subsection{The restriction theorem for polynomial representations }

The tensor restriction theorem generalises as follows. 

\begin{thm}[The restriction theorem in polynomial
representations] \label{thm:Restriction}
Let $P$ be a polynomial generic representation over the
finite field $K$ and for $i \in \NN$ let
$T_i \in P(V_i)$. Then there exist $i<j$ and a linear map $\phi:V_j \to
V_i$ such that $T_i=P(\phi) T_j$.
\end{thm}

We will use the term restriction also in this more general context, i.e.,
the conclusion of the theorem says that $T_i$ is a restriction of $T_j$.

\begin{re} \label{re:Circuit}
The condition that $P$ be polynomial cannot be dropped. For instance,
let $P$ be the functor that sends $V$ to $KV$. For each $n
\geq 3$ let $T_n \in P(K^{n-1})$ be the formal sum 
\[ T_n:=v_1 + \cdots + v_n \in P(K^{n-1}) \]
where $\{v_1,\ldots,v_n\} \in K^{n-1}$ is a {\em circuit}:
any $n-1$ of the $v_i$ are
a basis of $K^{n-1}$. We stress that in $P(K^{n-1})$
the $v_i$ are basis vectors, and $T_n$ is the sum of these basis
vectors. We claim that no $T_n$ is a restriction of any $T_m$ with $m
\neq n$. Indeed, if it were, then writing $T_m=v'_1 + \cdots + v'_m$,
there would be a linear map $K^{m-1} \to K^{n-1}$ that maps the circuit
$\{v_1',\ldots,v_m'\}$ to the circuit $\{v_1,\ldots,v_n\}$. By basic
linear algebra, such linear maps do not exist.
\end{re}

Corollary~\ref{cor:Main} generalises verbatim to polynomial 
representations, and so does Corollary~\ref{cor:WellOrdered}.

\begin{re} \label{re:Multivariate}
Versions of the Theorem~\ref{thm:Restriction} and its corollaries also hold for
{\em multivariate polynomial representations}, defined as functors
$P:\Vec^k \to \Vec$ for which $\dim_K P(K^{n_1},\ldots,K^{n_k})$
is a polynomial in $n_1,\ldots,n_k$. Indeed, given elements $T_i \in
P(V_i^{(1)},\ldots,V_i^{(k)})$ for $i=1,2,\ldots$, we can choose linear
injections $\iota_i^{(j)}$ from $V_i^{(j)}$ into a $U_i \in \Vec$
(which depends only on $i$), and linear surjections $\pi_i^{(j)}:U_i
\to V_i^{(j)}$ with $\pi_i^{(j)} \circ
\iota_i^{(j)}=\id_{V_i^{(j)}}$. Then define
\[ T_i':=P(\iota_i^{(1)},\ldots,\iota_i^{(k)})T_i \in
P(U_i,\ldots,U_i)=:Q(U_i) \]
where $Q$ is now a univariate polynomial generic representation.
Theorem~\ref{thm:Restriction} applied to $Q$ says that there exist $i<j$
and a linear map $\psi:U_j \to U_i$ such that
\[ Q(\psi)T_j'=P(\psi,\ldots,\psi)T_j'=T_i'. \]
We then have 
\[ P(\pi_i^{(1)} \circ \psi \circ
\iota_j^{(1)},\ldots,\pi_i^{(k)} \circ \psi \circ
\iota_j^{(1)})T_j=T_i, \]
as desired.
\end{re}

\subsection{Proof strategy: the polynomial method}

Rather than proving the restriction theorem for polynomial representations
directly, we will prove Noetherianity, corresponding to (1) in
Corollary~\ref{cor:Main}: if $P$ is a polynomial representation, and $X_1
\supseteq X_2 \supseteq \ldots$ are restriction-closed properties, then
$X_n=X_{n+1}$ for all sufficiently large $n$. 

To establish Noetherianity, we adapt the proof method of
\cite{Draisma17} for polynomial functors over {\em infinite fields} to our
current setting. This is far from straightforward. For instance, a
polynomial functor over an infinite field and its
coordinate ring both have a $\ZZ_{\geq
0}$-grading, whereas a polynomial representation over the
finite field $K$ and its
coordinate ring only have a grading by
$\{0,1,\ldots,|K|-1\}$. Nevertheless, after introducing the {\em degree}
$d$ of the polynomial representation $P$, we show that $P$ has a unique
minimal sub-representation $P_{>d-1}$ the quotient by which has degree
at most $d-1$. We think of $P_{>d-1}$ as the {\em top-degree part} of
$P$. We then take an irreducible sub-representation $R$ in $P_{>d-1}$, and
assume that the Noetherianity statement holds for $P/R$ and various other
polynomial representations that have the same top-degree part as $P/R$ and
are therefore in a lexicographic sense smaller than $P$. This means that
if $X_1 \supseteq X_2 \supseteq \ldots$ is a chain of restriction-closed
properties in $P$, then their projections $X'_1 \supseteq X'_2 \supseteq
\ldots$ in $P/R$ stabilise. Therefore, it suffices to prove Noetherianity
for properties $X \subseteq P$ that have a fixed projection $X' \subseteq
P/R$. Then, to prove that any property $X \subseteq P$ with projection
$X'$ is Noetherian, we think of each $X(V)$ as a {\em Zariski-closed
subset} of $P(V)$, i.e., as given by polynomial equations in the finite
vector space $P(V)$. We do induction on the minimal degree of an equation
that vanishes identically on $X$ but not on $X'$. Using {\em spreading operators}
we show that from such an equation we can construct many equations of
the same degree that are affine-linear in the $R$-direction. This allows
us to embed a certain subset of $X$ into a strictly smaller polynomial
functor, while on the complement of that subset a polynomial of strictly
smaller degree vanishes. Both subsets can therefore be
handled by induction.

We stress that this proof never actually looks at concrete tensors
or elements of $P(V)$---all reasoning uses polynomial
equations, and exploits the fact that every subset of a finite vector space is given
by polynomial equations. In this sense, the proof can be regarded an
instance of the polynomial method.

We remark that polynomial generic representations are not the same
thing as strict polynomial functors in the sense of Friedlander-Suslin
\cite{Friedlander97}. Roughly speaking, while former deal with sequences of
representations of the finite groups $\GL_n(K)$, the latter deal
with sequences of algebraic representations of the group schemes
$\GL_n$. Topological Noetherianity of strict polynomial functors,
over arbitrary rings with Noetherian spectrum and hence certainly over
finite fields, was established in \cite{Draisma20b}, using the techniques from
\cite{Draisma17}. However, strict polynomial functors have a
scheme structure built in and are therefore much more amenable to the techniques of
\cite{Draisma17} than
the polynomial generic representations that we study here. Furthermore,
even if one is interested only in polynomial generic representations
that come from strict polynomial functors by forgetting some of the
data---such as the functor $V \mapsto V^{\otimes d}$---our proof, in
which we mod out an irreducible subrepresentation $R$, requires that
one leaves the realm of these special representations.  This explains
the need for the new ideas developed in this paper.

\subsection{Further relations to the literature}

Restriction-closed properties of tensors are a rapidly expanding research
area. Here is a very small selection of recent research related to
our work.

In \cite{Karam22} it is proved, for various notions of rank including
ordinary tensor rank, slice rank, and partition rank, that a large tensor
of rank $r$ has a sub-tensor whose size depends only on $r$ and whose
rank is at least some function of $r$. For finite fields, this result also
follows from Corollary~\ref{cor:Main}, item (3)---in fact, by that item,
a subtensor of fixed size can be found of rank equal to $r$.  However,
Karam also finds an explicit formula for the size, while our theorem
does not give such a bound. It would be very interesting to see whether
{\em the proof} of our theorem could shed further light on such bounds.

In \cite{Cohen21} it is proved that over sufficiently large fields,
partition rank is bounded by a linear function of the analytic rank of
a tensor; and in \cite{Moshkovitz22} the condition on the field size is
removed at the cost of a polylogarithmic factor. This is the culmination
of many years of research by many authors into the relation between
bias and rank of tensors, starting with \cite{Green09} via polynomial
bounds in \cite{Milicevic19} and linear bounds for trilinear forms in
\cite{Adiprasito21}.  Using the proof of our tensor restriction theorem
and techniques from \cite{Draisma18b}, it is easy to recover the result
that partition rank is bounded from above by at least {\em some} function
of the analytic rank. However, again, our techniques do not yield bounds
that can compete with the state of the art.

In \cite{Conner21}, motivated by Strassen's asymptotic rank conjecture that says
that any {\em tight} tensor has the minimal possible asymptotic rank, the
authors study the geometry of various varieties of tensors, such as the
(closure of) the set of tight tensors. It would be interesting to study
these varieties from the perspective of this paper (over finite fields)
and from the perspective of \cite{Bik21} (over infinite fields). In
both cases, after a shift and a localisation, these varieties become of
the form a fixed finite-dimensional variety times an affine space that
depends on the size of the tensor. Over infinite fields, this follows
from the {\em shift theorem} in \cite{Bik21}, and over finite fields,
it follows from the weak shift theorem in this paper.

In \cite{Putman14} and \cite{Sam14}, the long-standing {\em
Lannes-Schwartz Artinian conjecture} was resolved, which says that any
finitely generated (not necessarily polynomial!) generic representation
$F:\Vec \to \Vec$ is Noetherian in the module sense: it satisfies the
ascending chain condition on subrepresentations. Dually, this means that
any descending chain of subrepresentations of $F^*:V^* \mapsto F(V)^*$
stabilises. Interpreting the elements of $F(V)$ as linear functions on
$F(V)^*$, one may interpret this as Noetherianity for {\em linear}
functorial subsets of $F^*$.  However, already for $F:V \mapsto K
\cdot V$, one can show that $F^*$ does not satisfy the descending chain
condition on {\em nonlinear} subsets. So topological Noetherianity as we
prove it seems restricted only to polynomial generic representations. It
would be nice to know a precise statement to this effect.  For instance,
is it true that the {\em only} generic representations for
which the restriction theorem holds are the polynomial representations?

In \cite{Snowden21}, Snowden extends many results about $\GL$-algebras
from \cite{Bik21} to modules over $\GL$-algebras equipped with a
compatible $\GL$-action.  Along the way, he also gives a proof of the
shift theorem that differs slightly from the proof in \cite{Bik21}, and
which uses the search for an element of weight $(1,\ldots,1)$ in a suitable
$\GL$-representation. This inspired the development of weight theory
in our current, different context in \S\ref{sec:Weights} and the idea
that a suitably spread out element in the vanishing ideal of a tensor
property would have weight $(1,\ldots,1)$---a key insight in the proof of the
embedding theorem in \S\ref{ssec:Embedding}.

\subsection{Organisation of this paper}

In Section~\ref{sec:PolyReps} we discuss the theory of generic polynomial
representations, including the definition of top-degree parts and shift
functors. In particular, we will see that the ring of functions on a
polynomial representation is itself a countable union of polynomial
representations.

In Section~\ref{sec:Weights}, we develop a partial analogue of the
classical weight theory for representations of group schemes $\GL_n$. This
includes the spreading operators alluded to above. Since functions on a
polynomial representation themselves live in a polynomial representation,
these spreading operators also act on functions.

In Section~\ref{sec:Noetherianity} we prove Noetherianity for polynomial
representations over the fixed finite field $K$, which implies the
restriction theorems for tensors and polynomial representations and
items (1)-(3) of Corollary~\ref{cor:Main}, both for tensors and for
polynomial representations, and also implies the existence
of a unique decomposition of a restriction-closed tensor property into 
irreducible such properties; see
Theorem~\ref{thm:Irreducibility}. We do so by first deriving Noetherianity
from an auxiliary result that we call the embedding theorem, since it is
the finite-field analogue of the embedding theorem in \cite{Bik21}. The
proof of the embedding theorem, then, is the heart of the paper. We also
derive from it a version of the shift theorem in \cite{Bik21}.

Finally, in Section~\ref{sec:FI} we use the theory of finitely generated
FI-modules to prove item (4) from Corollary~\ref{cor:Main}; as we have
seen, (5) is then a direct consequence.

\subsection*{Acknowledgments}

We thank Nate Harman and Andrew Snowden for useful discussions and for
references to relevant literature. 

\section{Polynomial generic representations} \label{sec:PolyReps}

Throughout the paper, $K$ is a fixed finite field, with $q$ elements. All
linear and multilinear algebra will be over $K$. We denote by $\Vec$ the
category of finite-dimensional $K$-vector spaces, and for $V \in \Vec$
we denote the dual space by $V^*$.

\subsection{Functions as polynomials} 

We introduce the ring of functions on a vector space; we
will also call this the coordinate ring. 

\begin{de}
Given $V \in \Vec$, we write $K[V]$ for the $K$-algebra of functions $V
\to K$. This has a natural algebra filtration
\[ \{0\}=K[V]_{\leq -1} \subseteq K[V]_{\leq 0} \subseteq K[V]_{\leq
1} \subseteq K[V]_{\leq 2} \subseteq \dots \]
where $K[V]_{\leq d}$ is the set of functions $f:V \to K$ for which
there exists an element of $\bigoplus_{e=0}^d S^e V^*$
that defines the function $f$.
\end{de}

We stress that $K[V]$ is an algebra of functions, not of polynomials. More
precisely, $K[V]$ is the quotient of the symmetric algebra $S V^*$
by the ideal generated by the polynomials $x^q-x$ as $x$ runs through
(a basis of) $V^*$. Since these polynomials are not homogeneous, $K[V]$
has no natural {\em grading}---but as seen above, it does have a natural
{\em filtration}.

Note further that $K[V]$ is a finite-dimensional $K$-vector
space, of dimension $q^{\dim(V)}$, the number of elements of $V$. 

\begin{de}
Given a basis $x_1,\dots,x_n$ of $V^*$, every element $f$ of $K[V]$ has a
unique representative polynomial in which all exponents of all variables
are $\leq q-1$; we will call this representative---which depends on the
choice of basis---the {\em reduced} polynomial representation for $f$
relative to the choice of coordinates.
\end{de}

The following lemma is immediate; the natural isomorphisms in it will
be interpreted as equalities throughout the paper.

\begin{lm}
For $V,W \in \Vec$ we have $K[V \times W] \cong K[V] \otimes K[W]$ via
the $K$-linear map map from right to left that sends $f \otimes g$ to
the function $(v,w) \mapsto f(v)g(w)$; this is a $K$-algebra isomorphism.

Similarly, the set of arbitrary maps $V \to W$ is canonically isomorphic to $K[V]
\otimes W$ via the $K$-linear map from right to left that sends $f
\otimes w$ to the function $v \mapsto f(v) \cdot w$.
\hfill $\square$
\end{lm}

Furthermore, we write $K[V]_0=K$ for the sub-$K$-algebra of constant
functions, and $K[V]_{>0}$ for the $K$-vector space spanned by all
functions that vanish at zero.

\subsection{Polynomial generic representations over $K$}

Recall Theorem~\ref{thm:Kuhn}, which characterises polynomial representations
among all generic representations. We will use the
following, alternative characterisation instead.

\begin{de} \label{de:PolRep}
A generic representation $P:\Vec \to \Vec$ is called {\em polynomial}
if there exists a $d$ such that for all $U,V \in \Vec$ the map
$P:\Hom(U,V) \to \Hom(P(U),P(V))$ lies in $K[\Hom(U,V)]_{\leq d} \otimes
\Hom(P(U),P(V))$. The minimal such $d \in \ZZ_{\geq -1}$ is called the
{\em degree} of $P$ and denoted $\deg(P)$.
\end{de}

Polynomial representations form an Abelian category, in which the
morphisms are natural transformations.

Any polynomial representation in the sense of Theorem~\ref{thm:Kuhn} is a
subquotient of a direct sum $T^{d_1} \oplus \cdots \oplus T^{d_k}$'s, and
this implies that it is polynomial in the sense of the definition above,
of degree at most the maximum of the $d_i$. In Remark~\ref{re:Kuhn},
we will see that, conversely, any generic representation that is
polynomial in the sense of the definition above is polynomial in the
sense of Theorem~\ref{thm:Kuhn}.

Every finite-degree {\em strict} polynomial functor $\Vec \to \Vec$
in the sense of Fried\-lan\-der-Suslin \cite{Friedlander97} gives rise
to a polynomial representation. But this forgetful functor
is not an equivalence of abelian categories. For instance, if $Q$ is a
strict polynomial functor of degree $d$ over $K$, then its $q$-Frobenius
twist is a polynomial functor of degree $dq$ over $K$ and hence not
isomorphic to $Q$. However, $Q$ and its $q$-Frobenius twist give rise to
the same generic representation.

\subsection{Schur algebras over $K$}

In spite of the discrepancy between strict polynomial functors
and polynomial generic representations, a version of the theorem by
Friedlander and Suslin that relates polynomial functors to representations
of the Schur algebra, does hold.

Fix a natural number $d$ and a $U \in \Vec$.  The composition map $\End(U)
\times \End(U) \to \End(U)$ gives rise, via pullback of functions,
to a $K$-linear map
\[ K[\End(U)]_{\leq d} \to K[\End(U)]_{\leq d} \otimes
K[\End(U)]_{\leq d}. \]
We write $A_{\leq d}(U):=K[\End(U)]_{\leq d}^*$. Dualising the map
above, we obtain a $K$-bilinear map 
\[ A_{\leq d}(U) \times A_{\leq d}(U) \to A_{\leq d}(U). \]
A straightforward computation, using the associativity of composition
of linear maps, shows that this turns $A_{\leq d}(U)$ into a unital,
associative algebra, with unit element $f \mapsto f(\id_U)$.

\begin{de}
The unital, associative algebra $A_{\leq d}(U)$ with the multiplication 
above is called the {\em Schur algebra} over $K$.
\end{de}

We remark that this is in fact a subalgebra of the Schur algebra in
\cite{Friedlander97}, which is the dual space to the space of degree
at most $d$ {\em polynomials} on $\End(U)$. Our Schur algebra consists
of only those linear functions that vanish on the ideal of polynomials
that define the zero function on $\End(U)$.

The Schur algebra comes with a homomorphism of monoids (not of
$K$-algebras) $\End(U) \to A_{\leq d}(U)$ defined by $\phi \mapsto
(f \mapsto f(\phi))$. This homomorphism is an embedding if $d \geq 1$.

\subsection{A finite-field analogue of the Friedlander-Suslin lemma}

Given a polynomial generic representation $P$ of degree at most $d$
and a vector space $U$, we turn $P(U)$ into an $A_{\leq d}(U)$-module
by a construction very similar to the construction of $A_{\leq d}(U)$:
first, the map
\[ \End(U) \times P(U) \to P(U), (\phi,p) \mapsto P(\phi)(p) \]
gives rise, via pull-back, to a $K$-linear map 
\[ P(U)^* \to K[\End(U)]_{\leq d} \otimes P(U)^*. \]
Dualising, we obtain a $K$-bilinear map  
\[ A_{\leq d}(U) \times P(U) \to P(U) \]
that turns $P(U)$ into a (unital) $A_{\leq d}(U)$-module. 

The following proposition is proved exactly as Friedlander-Suslin's
corresponding theorem, and it is also almost equivalent to
\cite{Kuhn94}{Proposition 4.10}.

\begin{prop} \label{prop:FriedlanderSuslin}
Fix a natural number $d$ and a $U \in \Vec$ of dimension at least $d$.
Then $P \mapsto P(U)$ is an equivalence of Abelian categories from the
category of polynomial generic representations $\Vec \to \Vec$ of degree
$\leq d$ to the category of $A_{\leq d}(U)$-representations.
\hfill $\square$
\end{prop}

\begin{re} \label{re:Kuhn}
It follows from this proposition that every polynomial
representation of degree $\leq d$ in the sense of
Definition~\ref{de:PolRep} has finite length. Therefore, it
is also polynomial in the sense of Theorem~\ref{thm:Kuhn}.
\end{re}

\begin{re} \label{re:Irred}
If $P$ is an irreducible polynomial generic representation, then for each
$U \in \Vec$, $P(U)$ is (zero or) an irreducible $\End(U)$-module. Indeed,
if $M$ were a nonzero proper submodule, then, for varying $V$, $Q(V):=\{p
\in P(V) \mid \forall \phi \in \Hom_\Vec(V,U) P(\phi)p \in M\}$ would
define a nonzero proper subrepresentation.
\end{re}

\subsection{Filtering a polynomial representation by degree}

A strict polynomial functor in the sense of Friedlander-Suslin has
a {\em grading} by degree. In contrast, we will see that a polynomial
generic representation only has a {\em filtration} by degree. One notable
exception is the degree-zero part of a polynomial
representation.

\begin{de} \label{de:DegreeZero}
Let $P:\Vec \to \Vec$ be a polynomial generic representation. Then we
define $P_0:\Vec \to \Vec$ by $P_0(V):=P(0)=:U$ for all $V \in \Vec$
and $P_0(\phi):=\id_U$ for all $\phi \in \Hom_\Vec(V,W)$. This is a direct
summand of $P$ in the Abelian category of polynomial generic
representations, called
the {\em degree-zero part} or {\em constant} part of $P$; we will also
informally say that $U$ is the constant part of $P$. The constant part
$P_0$ has a unique complement in $P$, namely,
\begin{align*} 
P_{>0}(V):=\{p \in P(V) \mid & P(0\cdot \id_V)p=0\}
\qedhere \end{align*}
\end{de}

A polynomial representation $P$ of degree at most some $e \leq d$ is
in particular a polynomial representation of degree at most $d$. On the
Schur algebra side, this inclusion of Abelian categories is made explicit
as follows: take a vector space $U$ of dimension at least $d$. Then the
inclusion $K[\End(U)]_{\leq e} \to K[\End(U)]_{\leq d}$
dualises to a
linear surjection $A_{\leq d}(U) \to A_{\leq e}(U)$. This surjection is
an algebra homomorphism, and hence if $M$ is a module over the latter
algebra, then it is also naturally a module over the former algebra.

This interpretation also shows {\em which} $A_{\leq d}(U)$-modules are
also $A_{\leq e}(U)$-modules, namely, those for which the kernel $I$
of the surjection acts as zero. Furthermore, if $M$ is an $A_{\leq
d}(U)$-module, and $N$ is an $A_{\leq d}(U)$-submodule of $M$, then
$M/N$ is an $A_{\leq e}(U)$-module if and only if $I \cdot (M/N)=0$,
i.e., if and only if $N$ contains the $A_{\leq d}(U)$-submodule $I \cdot
M$. We conclude that there is a unique inclusion-wise minimal $A_{\leq
d}(U)$-submodule $N$ of $M$ such that $M/N$ is an $A_{\leq e}(U)$-module,
namely, $N=I \cdot M$. 

By Proposition~\ref{prop:FriedlanderSuslin} we may translate this back to
polynomial representations: 

\begin{prop}
For any polynomial representation $P$ and any $e \in \ZZ_{\geq -1}$,
there is a unique inclusionwise minimal subrepresentation $Q$ such that
$P/Q$ is a polynomial representation of degree at most $e$.
\hfill $\square$
\end{prop}

\begin{de}
Let $P$ be a polynomial representation and let $e \in \ZZ_{\geq -1}$. 
The unique inclusionwise minimal subrepresentation $Q$ of $P$ such that $P/Q$
has degree $\leq e$ is denoted by $P_{>e}$. 
\end{de}

\begin{ex}
Suppose that $\cha K=2$.  Consider the polynomial representation $P:V
\mapsto S^2 V$ and the polynomial representation $Q$ that sends $V$
to the space of symmetric tensors in $V \otimes V$. Then $P$ has
as a subrepresentation the representation $R$ that maps $V$ to the
space of squares of elements of $V$, and this is the only nontrivial
subrepresentation unequal to $P$ itself. The quotient $P/R$ has degree
$2$, so $P_{>1}=P$. On the other hand, $Q$ has the subrepresentation $T$
that assigns to $V$ the set of skew-symmetric tensors in $V
\otimes V$---i.e., those in the linear span of tensors of
the form $u \otimes v - v \otimes u$ as $u,v$ range through
$V$---and the quotient $Q/T$ is isomorphic to $R$. Now if
$K=\FF_2$, then $R$ has degree $1$, so that $Q_{>1}=T$;
while if $K=\FF_2$, then $R$ has degree $2$, and therefore
$Q_{>1}=Q$.
\end{ex}

We clearly have 
\[ P=P_{>-1} \supseteq P_{> 0} \supseteq \dots \supseteq P_{>d}=\{0\}
\]
where $d=\deg(P)$; and a straightforward check shows that $P_{>0}$ in
this definition agrees with the direct complement $P_{>0}$ of $P_0$ in
Definition \ref{de:DegreeZero}. 

\begin{lm} \label{lm:Degree}
If $\alpha:P \to Q$ is a morphism in the Abelian category of polynomial
representations, then for each $e$ we have $\alpha(P_{>e}) \subseteq Q_{>e}$. 
\end{lm}

\begin{proof}
We compute
\[ P/\alpha^{-1}(Q_{>e}) \cong \im(\alpha)/(\im(\alpha) \cap Q_{>e})
\subseteq Q/Q_{>e}. \]
Since the latter representation has degree at most $e$, so does the first.
The defining property of $P_{>e}$ then implies that 
$P_{>e} \subseteq \alpha^{-1}(Q_{>e})$. This is equivalent to the
statement in the lemma. 
\end{proof}

\begin{lm} \label{lm:LargerSub}
Let $P$ be a polynomial representation, let $e \in \ZZ_{\geq -1}$, and let
$R$ be a subobject of $P_{>e}$, and hence of $P$. Then $(P/R)_{>e}
\cong P_{>e}/R$.
\end{lm}

\begin{proof}
By Lemma~\ref{lm:Degree}, the morphism $P \to P/R$ maps $P_{>e}$ into
$(P/R)_{>e}$, and its kernel on $P_{>e}$ is $R$, so that $P_{>e}/R$ maps
injectively into $(P/R)_{>e}$. To see that it also maps surjectively,
we note that
\[ (P/R)/(P_{>e}/R) \cong P/P_{>e} \]
has degree $\leq e$. Hence $P_{>e}/R$ contains $(P/R)_{>e}$ by
definition of the latter object.
\end{proof}

\subsection{Shifting} 

Just like a univariate polynomial can be shifted
over a constant, and then its leading term does not change, a
polynomial representation can be shifted over a constant vector space, and
we will see that its top-degree part does not change. 

\begin{de} \label{de:Shift}
Given a $U \in \Vec$ and a representation $P:\Vec \to \Vec$, we define the
representation $\Sh_U P$ by $(\Sh_U P)(V):=P(U \oplus V)$ and
$(\Sh_U P)(\phi):=P(\id_U \oplus \phi)$ for $\phi \in \Hom(V,W)$. We call
$\Sh_U P$ the {\em shift of $P$ by $U$}. 
\end{de}

If $P$ is polynomial of degree $\leq d$, then $\Sh_U P$ is also polynomial
of degree $\leq d$; below we will prove a more precise statement. 

We have a morphism $\alpha:P \to \Sh_U P$ in the Abelian category of
polynomial generic representations defined by
$\alpha_V=P(\iota_V):P(V) \to P(U \oplus V)$, where $\iota_V:V \to U
\oplus V$ is the inclusion $v \mapsto 0 + v$. Indeed, that
$(\alpha_V)_V$ is a morphism follows 
from the commutativity of the following diagram, for any $\phi \in
\Hom_\Vec(V,W)$:
\[ \xymatrix{
P(V) \ar[r]^{P(\iota_V)} \ar[d]_{P(\phi)} & P(U \oplus V)
\ar[d]^{P(\id_U \oplus \phi)} \\
P(W) \ar[r]_{P(\iota_W)} & P(U \oplus W),} \]
which in turn follows from the fact that $P$ is a
representation and that
\[ (\id_U \oplus \phi) \circ \iota_V = \iota_W \circ \phi. \]
Similarly, we have a morphism
$\beta:\Sh_U P
\to P$ defined by $\beta_V=P(\pi_V):P(U \oplus V) \to P(V)$, where
$\pi:U \oplus V \to V$ is the projection $u+v \mapsto v$. The relation
$\pi \circ \iota=\id_V$ translates to $\beta \circ \alpha=\id_P$. This
implies that $\Sh_U P$ is the direct sum of $\im(\alpha) \cong P$ and
the polynomial representation $Q:=\ker(\beta)$. 

The following lemma says, informally, that the top-degree part of a
polynomial representation is invariant under shifting. 

\begin{lm} \label{lm:ShiftTop}
Assume that $\deg(P)=d \geq 0$. Then $(\Sh_U P)_{>d-1} \cong P_{>d-1}$.
\end{lm}

\begin{proof}
Using the notation $\alpha$ and $\beta$ from above, we have
$\alpha(P_{>d-1}) \subseteq (\Sh_U P)_{>d-1}$ and $\beta((\Sh_U P)_{>d-1})
\subseteq P_{>d-1}$ by Lemma~\ref{lm:Degree}.  Combining these facts shows
that $\alpha$ maps $P_{>d-1}$ injectively into $(\Sh_U P)_{>d-1}$. To
argue that it also maps surjectively there, it suffices to show that
$(\Sh_U P)/\alpha(P_{>d-1})$ has degree $\leq d-1$.

To see this, we recall that $\Sh_U P=\im(\alpha) \oplus Q$, where
$Q=\ker(\beta)$. Accordingly, 
\[ (\Sh_U P)/\alpha(P_{>d-1}) \cong (\alpha(P)/\alpha(P_{>d-1})) \oplus
Q. \]
Here the first summand on the right is isomorphic to $P/P_{>d-1}$, hence of degree
$\leq d-1$. So it suffices to show that $Q$ has degree $\leq d-1$,
as well. Consider a vector $q \in Q(V)$ and a linear map $\phi \in
\Hom(V,W)$. Then we have 
\begin{align*} 
Q(\phi)(q)&=P(\id_U \oplus \phi)(q)=P(\id_U \oplus \phi)(q) - 
P(\id_U \oplus \phi)(\alpha_V(\beta_V(q))) \\
&= (P(\id_U \oplus \phi)-P(0_U \oplus \phi))(q) 
\end{align*}
where the second equality follows from $\beta_V(q)=0$ and the last
equality follows from the definition of $\alpha$ and $\beta$. Now,
for $\psi$ running through $\Hom(U \oplus V,U \oplus W)$, $P(\psi)$ can be
described by a polynomial map of degree at most $d$. If in this map we
substitute for $\psi$ the maps $\id_U \oplus \phi$ and $0_U \oplus
\phi$, respectively, we obtain the same degree-$d$ parts in $\phi$. Hence the map 
$\phi \mapsto P(\id_U \oplus \phi)-P(0_U \oplus \phi)$ is given by a
polynomial map of degree $\leq d-1$ in the entries of $\phi$. This
shows that $Q$ has degree $\leq d-1$, as desired. 
\end{proof}

\subsection{A well-founded order on polynomial
representations} \label{ssec:Order}

\begin{de}
Given polynomial representations $Q,P: \Vec \to \Vec$, we write $Q \preceq P$
if $Q \cong P$ or else for the largest $e$ such that $Q_{>e} \not \cong
P_{>e}$ the former is a quotient of the latter. We write $Q \prec P$
to mean $Q \preceq P$ and $Q \not \cong P$.
\end{de}

\begin{lm}
The relation $\preceq$ is a well-founded pre-order on polynomial
representations.
\end{lm}

\begin{proof}
Reflexivity is immediate.  To see transitivity, assume $R \preceq Q
\preceq P$. If one of the inequalities is an isomorphism, it follows
immediately that $R \preceq P$. Suppose that they are both not
isomorphisms. Let $e$ be maximal such that $Q_{>e} \not \cong P_{>e}$
and let $e'$ be maximal such that $R_{>e'} \not \cong Q_{>e'}$. If $e'
\geq e$, then $e'$ is maximal such that $R_{>e'} \not \cong P_{>e'}$, and
the former is a quotient of the latter. If  $e' < e$, then $e$ is maximal
such that $(Q_{>e} \cong )R_{>e} \not \cong P_{>e}$, and the former is
a quotient of the latter. In both cases, we find $R \prec P$, as desired.

To see that $\preceq$ is well-founded, suppose we had an infinite chain 
\[ P_1 \succ P_2 \succ \ldots. \] 
To each $P_i$ we associate a length sequence $\ell(P_i) \in \ZZ_{\geq
0}^{\{-1,0,1,2,\ldots\}}$, where $\ell(P_i)(e)$ is the length of any
composition chain of $(P_i)_{>e}$ in the abelian category of polynomial
representations; by Proposition~\ref{prop:FriedlanderSuslin} this length is finite.

Note that $\ell(P_i)(e)=0$ for $e \geq \deg(P_i)$, i.e., $\ell(P_i)$
has finite support. Now $P_i \succ P_{i+1}$ implies that 
$\ell(P_{i+1})$ is lexicographically strictly smaller
than $\ell(P_i)$. Since the lexicographic order on sequences with
finite support is a well-order, we arrive at a contradiction.
\end{proof}

We will intensively use the following construction.  

\begin{ex}
Let $P \neq 0$ be a polynomial representation of degree $d \geq 0$, and let
$R$ be an irreducible sub-object of $P_{>d-1}$. Let $U \in \Vec$ and
set $Q:=\Sh_U P$. By Lemma~\ref{lm:ShiftTop}, $R$ is also naturally
a sub-object of $Q_{>d-1}$, which in turn is a sub-object of $Q$. By
Lemma~\ref{lm:LargerSub}, we have $(Q/R)_{>d-1} \cong Q_{>d-1}/R$,
which in turn is $\cong P_{>d-1}/R$, a quotient of $P_{>d-1}$. Since
$P_{>e}=Q_{>e}=0$ for $e \geq d$, we conclude that $Q/R \prec P$.
\end{ex}

\subsection{The coordinate ring of a polynomial
representation}

\begin{de}
Let $P:\Vec \to \Vec$ be a polynomial representation. We define $K[P]$ as the
contravariant functor from $\Vec$ to $K$-algebras that assigns to $V$
the ring $K[P(V)]$ and to a linear map $\phi:V \to W$ the pullback
$P(\phi)^\#: K[P(W)] \to K[P(V)]$. We call $K[P]$ the 
the {\em coordinate ring of $P$}. 
\end{de}

Note that $P(\phi)^\#$ is an algebra homomorphism; this is
going to be of crucial importance in \S\ref{ssec:Embedding}. 
The coordinate ring comes with a natural ring filtration:
\[ \{0\}=K[P]_{\leq -1} \subseteq K[P]_{\leq 0} \subseteq K[P]_{\leq
1} \subseteq K[P]_{\leq 2} \subseteq \dots \]
where $K[P]_{\leq e}$ assigns to $V$ the space $K[P(V)]_{\leq e}$. 

\begin{lm} \label{lm:Dual}
If $P$ is a polynomial representation of degree $\leq d$, then $V^* \mapsto
K[P(V)]_{\leq e}$ is a polynomial representation of degree $\leq d \cdot e$.
\end{lm}

\begin{proof}
This representation assigns to a linear map $\phi:V^* \to W^*$ the restriction
of the pullback $P(\phi^*)^\# : K[P(V)] \to K[P(W)]$ to $K[P(V)]_{\leq
e}$. Since $P(\phi^*)$ is a linear map, this pullback does indeed map
$K[P(V)]_{\leq e}$ into $K[P(W)]_{\leq e}$, and it does so via a linear
map that is polynomial of degree $\leq e$ in $P(\phi^*)$, hence of degree
$\leq d \cdot e$ in $\phi^*$, which in turn depends linearly on $\phi$.
\end{proof}

\begin{ex} \label{ex:Contravariant}
Let $P=S^2$ and assume $|K|>2$. Take $V=K^n$ with basis $e_1,\ldots,e_n$,
so that $P(V)$ has basis $e_i e_j$ with $i \leq j$. For $k>l$ distinct,
let $g_{kl}(s) \in \End(V)$ be the matrix with $1$'s on the diagonal,
an $s$ on position $(k,l)$, and zeros elsewhere. We have
\begin{align*} 
P(g_{kl}(s)) \sum_{i \leq j} a_{ij} e_i e_j &= \sum_{i \leq
j} a_{ij}(g_{kl}(s) e_i)(g_{kl}(s) e_j)\\ 
&= \sum_{i \leq j} a_{ij} (e_i + \delta_{il} s e_k)(e_j + \delta_{jl} s e_k) \\
&= \sum_{i \leq j} a_{ij} (e_i e_j + s(\delta_{jl}
e_i e_k + \delta_{il} e_k e_j) + s^2 \delta_{il} \delta_{jl} e_k^2)\\
&= \left(\sum_{i \leq j} a_{ij} e_i e_j \right)
+ s \left(
\sum_{i \leq l} a_{il} e_i e_k 
+ 
\sum_{j \geq l} a_{lj} e_k e_j 
\right)
+ s^2 a_{ll} e_k^2
\end{align*}
Observe that by acting with $g_{kl}$ on (linear
combinations of) the basis vectors $e_i e_j$, indices $l$ 
either remain the same or turn into indices $k$.

We now look at the dual. Let $\{x_{ij} \mid i \leq j\}$ be the basis of
$P(V)^*$ dual to the given basis of $P(V)$. Then, for instance, for $l
< i < k$ we have
\[ P(g_{kl}(s))^\# x_{ik}=x_{ik} + s x_{li}, \]
as can be seen by taking the coefficient of $e_i e_k$ in the expression
above. We observe here that indices $k$ either remain the same
or turn into indices $l$. We can also write the above as
\[ P(g_{lk}(s)^T)^\# x_{ik}=x_{ik} + s x_{li}. \]
Note that $P(g)^\#$ is contravariant in $g$, and hence $P(g^T)^\#$
is again covariant. This explains the $V^*$ in Lemma~\ref{lm:Dual}.
\end{ex}

\section{Weight theory} \label{sec:Weights}

In the representation theory of the group scheme $\GL_n$, the weight
space decomposition of a representation, i.e., its decomposition as a
module over the subgroup of diagonal matrices, is of crucial importance.
For the finite group $G=\GL_n(K)$ with $K=\FF_q$, it was already
observed in \cite[Page 129, Remark]{Steinberg16} that weights alone
do not suffice to distinguish the roles played by various vectors in
a representation. The example given there is that the highest weight
vector $1$ in the trivial $\GL_n(K)$-representation $K^1$ and the highest
weight vector $e_1^{q-1}$ in the $(q-1)$st symmetric power $S^{q-1}K^n$
of the standard representation both have weight $(0,\ldots,0)$.  It is
explained there how to act with elements of the group algebra of $G$
to distinguish the two.

In this particular example, we {\em can} distinguish these vectors by
extending the action of (diagonal matrices in) $\GL_n(K)$ to (diagonal
matrices in) $\End(K^n)$, as we do in the context of generic polynomial
representations: the first heighest weight vector then has weight
$(0,\ldots,0)$, while the latter has weight $(q-1,0,\ldots,0)$---see
the definitions below. However, the vector $e_1^{\otimes q}$ in the
$q$-th tensor power cannot be distinguished from the vector $e_1$ in
the standard representation via diagonal matrices. Therefore we, too,
will act with suitable elements of the monoid algebra of $\End(K^n)$
to get a better grasp on weight vectors. However, our focus will not be
on {\em highest} weight vector; rather, we will look for {\em middle}
weight vectors, i.e., weight vectors whose weight is
maximally spread out in a sense that we will make precise
below.

We are by no means the first to study weights in this
context. For instance, they also feature as {\em reduced weights} in
\cite{Kuhn94b}. However, the procedure of maximally spreading out weight
that we introduce below does seem to be new.

\subsection{Multiplicative monoid homomorphisms $K \to K$}

A monoid homomorphism $(K, \cdot) \to (K,\cdot)$ is a map $\phi:K \to
K$ with $\phi(1)=1$ and $\phi(ab)=\phi(a) \phi(b)$ for all $a,b \in
K$. In particular, $\phi$ restricts to a group homomorphism from the
multiplicative group $K^\times:=K \setminus \{0\}$ to itself. Since
$K^\times$ is cyclic, say with generator $g$, the monoid homomorphism
$\phi$ is uniquely determined by its values on $g$ and on $0$. Write
$\phi(g)=g^e$ for a unique exponent $e \in \{1,\ldots,q-1\}$. If $e
\neq q-1$, so that $\phi(g) \neq 1$, then $\phi(0)$ is forced to be
$0$, since otherwise $\phi(g) \cdot \phi(0)$ does not equal $\phi(g
\cdot 0)=\phi(0)$. If $e = q-1$, then there are two possibilities for
$\phi(0)$, namely, $\phi(0)=1$ and $\phi(0)=0$. In the first case,
we will denote $\phi$ by $c \mapsto c^0$, and in the second case, we
denote by $c \mapsto c^{q-1}$. The following is now straightforward.

\begin{lm} \label{lm:MonHomIso}
The monoid of monoid homomorphisms $K \to K$ is isomorphic to the monoid
$\{0,\ldots,q-1\}$ with operation $i \oplus j$ defined by $i \oplus j=i+j$
if $i+j \leq q-1$ and $i \oplus j = i+j-(q-1)$ otherwise.
\end{lm}

Note that this monoid is not cancellative, since $0 \oplus j = (q-1)
\oplus j$ for all $j \in \{1,\ldots,q-1\}$.  Nevertheless, it will be
convenient to have a notation for subtracting elements in the following
sense: for $i \in \{1,\ldots,q-1\}$ and $j \in \{0,\ldots,q-1\}$ we write
$i \omin j$ for the unique element in $\{1,\ldots,q-1\}$ that equals $i-j$
modulo $q-1$.

\subsection{Acting with diagonal matrices}

Let $P:\Vec \to \Vec$ be a polynomial representation and set $V:=K^n$, so
that we may identify $\End(V)$ with the space of $n \times n$-matrices.
Then the monoid $\End(V)$ acts linearly on $P(V)$ via the map $\End(V)
\to \End(P(V)), \phi \mapsto P(\phi)$, and hence so does its submonoid
$D_n \subseteq \End(V)$ of diagonal matrices.

\begin{lm} \label{lm:KMon}
We have
\[ P(V)=\bigoplus_{\chi:D_n \to K} P(V)_\chi \]
where $\chi$ runs over all monoid homomorphisms $(D_n,\cdot) \to
(K,\cdot)$ and where 
\[ P(V)_\chi:=\{p \in P(V) \mid \forall \phi \in D_n:
P(\phi)p=\chi(\phi)p \}. \]
\end{lm}

\begin{proof}
Each element $\phi \in D_n$ satisfies $\phi^q=\phi$, and therefore
also $P(\phi)^q=P(\phi^q)=P(\phi)$. Consequently, $P(\phi)$ is a root
of the polynomial $h=T\cdot (T^{q-1}-1) \in K[T]$.  This polynomial is
square-free, so that $P(\phi)$ is diagonalisable over a separable closure
of $K$. But also the eigenvalues of $P(\phi)$ are roots of $h$, i.e.,
elements of $K$, so $P(\phi)$ is diagonalisable over $K$.  Moreover,
all elements of $D_n$ commute, and therefore so do all elements of
$P(D_n)$. Hence the latter are all simultaneously diagonalisable.
We therefore have
\[ P(V)=\bigoplus_{\chi:D_n \to K} P(V)_\chi \]
where, {\em a priori}, $\chi$ runs through all maps $D_n \to K$. 

Now if $P(V)_\chi \neq 0$, then it follows that
$\chi(\diag(1,\ldots,1))=1$ and $\chi(\phi
\psi)=\chi(\phi) \chi(\psi)$, i.e., $\chi$ is a monoid homomorphism
$D_n \to K$. 
\end{proof}

Note that monoid homomorphisms $D_n \to K$ can be naturally
identified with $n$-tuples of monoid homomorphisms $K \to K$, and
hence, by Lemma~\ref{lm:MonHomIso}, with elements of $\{0,\ldots,q-1\}^n$.
Explicitly, $\chi$ is identified with the tuple $(a_1,\ldots,a_n)$
if $\chi(\diag(t_1,\ldots,t_n))=t_1^{a_1} \cdots t_n^{a_n}$ for all
$(t_1,\ldots,t_n) \in K^n$.

In analogy with the theory of representations of algebraic groups, we will
use the word {\em weight} for monoid homomorphisms $\chi:D_n \to K$, and
we call a vector in $P(V)_\chi$ a {\em weight vector} of weight $\chi$.

We use the notation $\oplus$ also in this context: if
$\chi,\mu \in \{0,\ldots,q-1\}^n$ are weights, then $\chi \oplus
\mu$ is their componentwise sum with respect to $\oplus$. Note that
\[ (\chi \oplus
\mu)(\diag(t_1,\ldots,t_n))=\chi(\diag(t_1,\ldots,t_n)) \cdot
\mu(\diag(t_1,\ldots,t_n)).\]

\begin{ex}
If $U$ is the subspace of $V$ spanned by the first $k$ basis vectors,
then $P(U)$, regarded as a subspace of $P(V)$, is the direct sum of all
$P(V)_\chi$ where $\chi$ runs over the characters in $\{0,\dots,q-1\}^k
\times \{0\}^{n-k}$. In particular, the constant part of $P$ is
$P(0)=P(V)_{(0,\ldots,0)}$.
\end{ex}

\begin{lm} \label{lm:BoundSupport}
Let $\chi=(a_1,\ldots,a_n) \in \{0,\ldots,q-1\}^n$ be a weight such that
$P(K^n)_\chi$ is nonzero. Then $\sum_i a_i$ is at most
$\deg(P)$. 
\end{lm}

\begin{proof}
Choose a nonzero $p \in P(K^n)_\chi$. Then
$P(\diag(t_1,\ldots,t_n))p=t_1^{a_1} \cdots t_n^{a_n} p$, and we
note that $t_1^{a_1} \cdots t_n^{a_n}$ is a reduced polynomial in
$t_1,\ldots,t_n$. On the other hand, $P(\diag(t_1,\ldots,t_n))$ can
be expressed as a reduced polynomial of degree at most $\deg(P)$ in
$t_1,\ldots,t_n$ with coefficients that are linear maps $P(K^n) \to
P(K^n)$. Evaluating this at $p$ yields a reduced polynomial of degree
at most $\deg(P)$ in $t_1,\ldots,t_n$ whose coefficients are elements of
$P(K^n)$. But we already know which polynomial that is, namely $t_1^{a_1} \cdots t_n^{a_n}
p$. Hence $\sum_i a_i \leq \deg(P)$.
\end{proof}

\subsection{Acting with additive one-parameter subgroups}
\label{ssec:Additive}

Let $P:\Vec \to \Vec$ be a polynomial representation, $n \in \ZZ_{\geq 2}$,
and $i,j \in [n]$ distinct. Then we have a one-parameter subgroup
\[ g_{ij}:(K,+) \to \GL_n(K),\  g_{ij}(s):=I+s E_{ij}, \]
where $E_{ij}$ is the matrix with zeroes everywhere except for a $1$
in position $(i,j)$. For $b=0,\ldots,q-1$ we define the linear map
$F_{ij}[b]:P(K^n) \to P(K^n)$ by
\[ F_{ij}[b]p=\text{ the coefficient of $s^b$ in
}P(g_{ij}(s))p, 
\]
where we write $P(g_{ij}(s))p$ as a reduced polynomial in $s$ with
coefficients in $P(K^n)$.

\begin{lm}
For any sub-representation $Q$ of $P$, the linear space $Q(K^n)$ is
stable under $F_{ij}[b]$. 
\end{lm}

\begin{proof}
Let $p \in Q(K^n)$. Then for all $s \in K$ the element
\[ P(g_{ij}(s))p=F_{ij}[0]p + s F_{ij}[1]p + \cdots +
s^{q-1} F_{ij}[q-1]p \]
lies in $Q(K^n)$. The Vandermonde matrix $(s^e)_{s \in K, e \in
\{0,\ldots,q-1\}}$ is invertible, and this implies that each of the
$F_{ij}[e]p$ above are linear combinations of the
$P(g_{ij}(s))p$, and therefore in $Q(K^n)$.
\end{proof}

\begin{lm} \label{lm:OneParameter}
Let $p \in P(K^n)$ be a weight vector of weight $a=(a_1,\ldots,a_n)$,
let $b \in \{0,\ldots,q-1\}$, and set $\tilde{p}:=F_{ij}[b]p$. 
Then we have the following.
\begin{enumerate}
\item $\tilde{p}=p$ for $b=0$;
\item if $a_j=0$, then $\tilde{p}=0$ for $b \neq 0$;
\item if $0<a_j \neq b$, then $\tilde{p}$ is a weight vector of
weight $a \omin (b e_j) \oplus (b e_i)$;
\item if $0<a_j = b$, then $\tilde{p}$ is a sum of a weight
vector of weight 
\[ a \omin b e_j \oplus b e_i 
= (a_1,\ldots,a_i \oplus b,\ldots,q-1,\ldots,a_n) \]
and a weight vector of weight 
\[ a - b e_j \oplus b e_i 
= (a_1,\ldots,a_i \oplus b,\ldots,0,\ldots,a_n). \]
\end{enumerate}
\end{lm}

\begin{proof}
We write 
\[ P(g_{ij}(s))p=p_0 + sp_1 + \cdots + s^{q-1}p_{q-1}. \]
By setting $s$ equal to zero we obtain
$P(g_{ij}(0))p=P(\id_{K^n})p=p$ on the left-hand side, and
$p_0$ on the right-hand side. This proves the first item.

If $a_j=0$, then
\[ P(\diag(1,\ldots,1,0,1\ldots,1)) p = p \]
where the $0$ is on position $j$. Therefore
\[ P(g_{ij}(s)) p = P(g_{ij}(s)
\diag(1,\ldots,1,0,1,\ldots,1)) p =
P(\diag(1,\ldots,1,0,1,\ldots,1)) p \]
does not depend on $s$ and hence $F_{ij}[b] p =0$ for $b
\neq 0$. 

We now assume $a_j>0$. We have $F_{ij}[b] p=p_b$. To determine the
weight(s) appearing in $p_b$, we act on $p_b$ with diagonal matrices. For
$t=(t_1,\ldots,t_n) \in K^n$ and $t_j \neq 0$ we have
\[ 
\diag(t_1,\ldots,t_n) \cdot g_{ij}(s)
= g_{ij}(t_i s t_j^{-1}) \cdot \diag(t_1,\ldots,t_n) \]
and therefore 
\begin{align*} 
\sum_{d=0}^{q-1} s^d P(\diag(t_1,\ldots,t_n)) p_d 
&= P(\diag(t_1,\ldots,t_n) g_{ij}(s)) p \\
&= P(g_{ij}(t_ist_j^{-1}) \diag(t_1,\ldots,t_n)) p \\
&= t_1^{a_1} \cdots t_n^{a_n} \cdot P(g_{ij}(t_i s
t_j^{-1})) p \\
&= t_1^{a_1} \cdots t_n^{a_n} \cdot \sum_{d=0}^{q-1} (t_i s t_j^{-1})^d p_d. 
\end{align*}
Comparing coefficients of $s^b$, we find
\[ P(\diag(t))p_b=t^{a-be_j+be_i} p_b \]
for all $t \in K^{j-1} \times K^\times \times K^{n-j}=:D$. Hence
$p_b$ is a linear combination of weight vectors with weights
that on $D$ agree with the weight $a \omin b e_j \oplus b
e_i$. If $a_j \neq b_j$,
there is only one such weight, namely, $a \omin b e_j \oplus
b e_i$. If $a_j=b$, then there are two such weights, namely,
$a \omin b e_j \oplus b e_i$ and $a - be_j \oplus b e_i$. 
\end{proof}

\subsection{Spreading out weight}

Retaining the notation from \S\ref{ssec:Additive}, 
suppose we are given a nonzero weight vector $p \in P(K^n)$ of weight
$(a_1,\ldots,a_n) \in \{0,\ldots,q-1\}^n$ and a $j \in [n]$ with $a_j >
0$. We construct vectors $\tilde{p} \in P(K^{n+1})$ by identifying
$p$ with $P(\iota)p$, where $\iota:K^n \to K^{n+1}$ is the embedding
adding a $0$ in the last position.  Then $p$ is a vector of weight
$a=(a_1,\ldots,a_n,0)$ in $P(K^{n+1})$, and we compute
\[ \tilde{p}:=F_{n+1,j}[b] p \]
for various $b$. The vector $\tilde{p}$ is guaranteed to be nonzero for
at least two values of $b$, namely, for $b=0$ (in which case
$\tilde{p}=p$), and, as we will now see, for $b=a_j$.
Indeed, in the latter case, by Lemma~\ref{lm:OneParameter}, 
$\tilde{p}$ is the sum of a weight vector $\tilde{p}_0$ of weight $a-a_j
e_j + a_j e_{n+1}$ and a weight vector $\tilde{p}_1$ of
weight $a + (q-1-a_j) e_j + a_j e_{n+1}$.

\begin{lm} \label{lm:Swap}
In the case where $b=a_j$, we have $\tilde{p}_0=P((j,n+1))
p$, where $(j,n+1)$ is short-hand for the permutation matrix corresponding
to the transposition $(j,n+1)$. 
\end{lm}

\begin{proof}
The vector $\tilde{p}_0$ is obtained by applying $P(\pi_j)$ to
$\tilde{p}$, where $\pi_j$ is the projection $K^{n+1} \to K^{n+1}$
that sets the $j$-th coordinate to zero. Furthermore, we have
$P(\pi_{n+1})p=p$, where $\pi_{n+1}$ sets the $(n+1)$st coordinate to
zero. We can then compute $\tilde{p}_0$ as the coefficient of $s^{a_j}$ in
\begin{align*} 
P(\pi_j) P(g_{n+1,j}(s)) p &= P(\pi_j g_{n+1,j}(s)
\pi_{n+1}) p\\ 
&= P((j,n+1))P(\diag(1,\ldots,1,s,1,\ldots,1,0)) p \\
&= P((j,n+1))s^{a_j} p.  \qedhere
\end{align*}
\end{proof}

If either $F_{n+1,j}[b] p \neq 0$ for some $b \neq 0,a_j$ or if
$F_{n+1,j}[a_j]p \neq P((j,n+1)) p$, then we find a new vector $p'$ in
the subrepresentation of $P$ generated by $p$ whose weight has strictly more
nonzero entries---we have {\em spread out the weight} of $p$.

\begin{de}
A nonzero weight vector $p \in P(K^n) \subseteq P(K^{n+1})$
of weight $a \in \{0,\ldots,q-1\}^n $ is called 
{\em maximally spread out} if for all $j \in [n]$
with $a_j>0$ we have
\[ P(g_{n+1,j}(s)) p = p + s^{a_j} P((j,n+1)) p. \qedhere \]
\end{de}

\begin{prop} \label{prop:MaxSpread}
For any nonzero polynomial representation $P$, there exist an $n$ and a nonzero
weight vector $p \in P(K^n)$ that is maximally spread out.
\end{prop}

\begin{proof}
Let $p \in P(K^m)$ be a nonzero weight vector. As long as $p$ is not
maximally spread out, by the above discussion we can replace $p$ by a
nonzero weight vector in $P(K^{m+1})$ whose weight has strictly more nonzero
entries. But by Lemma~\ref{lm:BoundSupport}, the number of
nonzero entries is bounded from above by $\deg(P)$. Hence
this process must terminate, with a maximally spread out
vector. 
\end{proof}

\begin{ex} \label{ex:NotSpread}
It is not true that every polynomial representation is generated by
its maximally spread out vectors. Consider, for instance,
$K$ of characteristic $2$ and the representation $Q$ that sends $V$ to the space of symmetric tensors
in $V \otimes V$. The weight vectors in $Q(K^n)$ are of the forms $e_i
\otimes e_i$ and $e_i \otimes e_j + e_j \otimes e_i \in Q(K^n)$ with $i
\neq j$.  Only the latter are maximally spread out. But they generate
the sub-representation of $Q$ consisting of all skew-symmetric tensors
in $V \otimes V$.
\end{ex}

\subsection{The prime field case}

In this section we assume that $q$ is a prime, so that $K$ is a prime
field. We retain the notation from above.

\begin{de}
Let $\iota:K^n \to K^{n+1}$ be the standard embedding and
$F_{n+1,j}=F_{n+1,j}[1]:P(K^n) \to P(K^{n+1})$ be the operator that
sends $p$ to the coefficient of $s^1$ in $P(g_{n+1,j}(s) \circ \iota)(p)$.
\end{de}

\begin{lm} \label{lm:Indep}
Assume that $K$ is a prime field. Then the operator $F_{n+1,j}:P(K^n)
\to P(K^{n+1})$ is injective on the direct sum of all weight spaces
corresponding to weights $\chi=(a_1,\ldots,a_n)$ with $a_j
> 0$, and it is zero on the weight spaces corresponding to
weights with $a_j=0$. 
\end{lm}

\begin{proof}
The last part follows immediately from
Lemma~\ref{lm:OneParameter}; we now prove the first part.
The operator $F_{n+1,j}$ maps the weight space of $\chi$ into that of
$\chi \omin e_j + e_{n+1}$ if $a_j>1$ and into the sum of the weight
spaces with weights $\chi - e_j + e_{n+1}$ and $\chi \omin e_j + e_{n+1}$
if $a_j=1$. Since these weights are distinct for distinct $\chi$, it
suffices to show that $F_{n+1,j}$ is injective on a single weight space,
corresponding to the weight $(a_1,\ldots,a_n)$, where
$a_j>0$. Let $p$ be a nonzero vector in this weight space. 

Define $\phi:K^{n+1} \to K^{n}$ by
\[
\phi(c_1,\dots,c_{n+1}):=(c_1,\ldots,c_j+c_{n+1},\ldots,c_n).
\]
We then have 
\[ \phi \circ g_{n+1,j}(s) \circ \iota=\diag(1,\dots,1+s,\dots,1) \]
and therefore 
\[ P(\phi) P(g_{n+1,j}(s)) P(\iota) p = (1+s)^{a_j} \cdot p. \]
The coefficient of $s^1$ in the latter expression is $a_j \cdot p$, which
is nonzero since $a_j<q$ and $q$ is prime. That coefficient is also equal
to $P(\phi)\tilde{p}$, where $\tilde{p}:=F_{n+1,j}p$. Hence
$\tilde{p} \neq 0$.
\end{proof}

By Lemma~\ref{lm:Indep}, if $\chi=(a_1,\ldots,a_n)$ with $a_j>1$, then
$F_{n+1,j}$ maps $P(K^n)_\chi$ injectively into $P(K^{n+1})_{\chi'}$,
where $\chi'=\chi-e_j+e_{n+1}$. On the other hand, if $a_j=1$, then by
Lemma~\ref{lm:Swap}, $F_{n+1,j}$ followed by projection to the weight
space of $\chi'=(a_1,\ldots,0,\ldots,a_n,1)$ agrees on $P(K^n)_\chi$
with the map $P((n+1,j))r$, which of course we already knew is injective.

\begin{ex}
We note that Lemma~\ref{lm:Indep} is false for non-prime fields. Indeed,
take $K=\FF_4$ and $P=S^2$. Consider the element $p:=e_1^2 \in P(K^1)$,
of weight $(2)$. Now $g_{21}(s)p=(e_1+s e_2)^2=e_1^2 + s^2 e_2^2$,
and hence $F_{2,1}p=0$. On the other hand, if $K=\FF_2$, then $s^2=s$,
and $F_{21}p=e_2^2$.
\end{ex}

\begin{re} \label{re:SpreadOut}
Note that, as a consequence of the lemma, if a weight vector $p$
of weight $(a_1,\ldots,a_n)$ is maximally spread out, then $a_j \in
\{0,1\}$ for all $j$, and moreover $F_{n+1,j}p=P((n+1,j))p$ for all
$j$ with $a_j=1$. Indeed, if $a_j>1$, then $F_{n+1,j} p$ is a weight
vector of weight $(a_1,\ldots,a_j-1,\ldots,a_n,1)$, and if $a_j=1$ but
$F_{n+1,j}p \neq P((n+1,j))p$, then the left-hand side has a component
of weight $(a_1,\ldots,q-1,\ldots,a_n,1)$. In either case, $p$ was not
maximally spread out.
\end{re}

\begin{re}
For $\phi \in \Hom_\Vec(V,W)$ and $P$ a generic
polynomial representation and $p \in P(V)$, we will sometimes just write
$\phi p$ instead of $P(\phi)p$. For instance, we will do this for
$\phi=g_{n+1,j}(s)$. The advantage of this is that we do not need to make
explicit in which polynomial polynomial representation we are computing.
\end{re}

\section{Noetherianity for polynomial representations}
\label{sec:Noetherianity}

\subsection{The main result}

We recall that the tensor testriction theorem, its generalisation to
polynomial representations, and Corollary~\ref{cor:Main}, concern
restriction-closed properties. We will simply use the term
{\em subset} for such a property:

\begin{de} \label{de:Subset}
Let $P$ be a polynomial representation. A {\em subset} of $P$ is the data of a
subset $X(V)$ of $P(V)$ for every $V \in \Vec$, subject to the condition
that for all $V,W,\phi \in \Hom_\Vec(V,W)$, $P(\phi)X(V) \subseteq X(W)$.
\end{de}

\begin{thm}[Noetherianity] \label{thm:Noetherianity}
Let $P$ be a polynomial representation over the finite field $K$. Then
any descending chain
\[ P \supseteq X_1 \supseteq X_2 \supseteq \ldots \]
of subsets stabilises.
\end{thm}

\subsection{Irreducible decomposition of restriction-closed
tensor properties}

Before proceeding with the proof of Noetherianity, we deduce from it
the fact that any subset of $P$ admits a unique decomposition into
irreducible subsets.

\begin{de}
Let $P$ be a polynomial generic representation over the
finite field $K$ and let $X$ be a subset of $P$. We call $X$
{\em irreducible} if $X(0) \neq \emptyset$ and if whenever
$X_1,X_2$ are subsets of $P$ such that $X(V)=X_1(V) \cup
X_2(V)$ holds for all $V \in \Vec$, it follows that $X=X_1$ or $X=X_2$. 
\end{de}

\begin{thm} \label{thm:Irreducibility}
For any subset $X$ of a polynomial representation $P$ over
the finite field $K$, there is a unique decomposition 
\[ X=X_1 \cup \ldots \cup X_k \]
where all $X_i$ are irreducible and none is contained in any other.
\end{thm}

\begin{proof}
This is an immediate consequence of
Noetherianity~\ref{thm:Noetherianity}, and the proof is 
identical to the proof that any Noetherian topological space
admits a unique decomposition into irreducible closed subspaces.
\end{proof}

For an instructive example, we need the following lemma.

\begin{lm} \label{lm:Pure}
Suppose that $P(0)=0$. Then the subset $X:=P$ is irreducible in the sense
above. 
\end{lm}

We note that the requirement that $P(0)$ be zero is necessary for
irreducibility; otherwise, one can take any partition $S_1 \sqcup S_2$
of the finite set $P(0)$ into two nonempty parts, define $X_i(V)$ to
be the set of elements in $P(V)$ that map into the $S_i$, and note that
$P=X_1 \cup X_2$.

\begin{proof}
Suppose that $X=X_1 \cup X_2$ where $X_i \subsetneq X$
for $i=1,2$. Then $X_1$ has at least one forbidden restriction $T_1
\in P(V_1)$, and $X_2$ has at least one forbidden restriction $T_2 \in
P(V_2)$. Let $\iota_i: V_i \to V_1 \oplus V_2$ be the canonical inclusion,
and write $T:=P(\iota_1)(T_1) + P(\iota_2)(T_2)$. Let $\pi_i:V_1 \oplus
V_2 \to V_i$ be the projection. Then $P(\pi_1 \circ \iota_2)=P(0_{V_2
\to V_1})$, which is the zero map since it factors via $P(0)=0$; and
similarly $P(\pi_2 \circ \iota_1)=0$. We conclude that 
\[ T_1=P(\pi_1 \circ \iota_1)T_1 = P(\pi_1)(P(\iota_1)T_1 +
P(\iota_2)T_2)=P(\pi_1)T, \]
so $T_1$ is a restriction of $T$. Similarly, $T_2$ is a restriction
of $T$. It follows that $T$ lies neither in $X_1(V_1 \oplus V_2)$ nor
in $X_2(V_1 \oplus V_2)$, a contradiction. Hence $X$ is irreducible as
claimed. 
\end{proof}

\begin{ex}
For any $r \in \RR_{\geq 0}$, let $X_r$ be the locus in $T^d$ where the
{\em partition rank} is at most $r$, and let $Y_r$ be the locus in $T^d$
where the {\em analytic rank} is at most $r$.

For $X_r$ (with $r$ an integer) it is easy to write down a decomposition
into irreducible subsets: for any of the $2^{d-1}-1$ unordered partitions
$\{I,J\}$ of $[d]$ into two nonempty sets, choose a number $r_{\{I,J\}}$
such that these numbers add up to $r$. This choice gives a natural map
\[ \prod_{\{I,J\}} (T^{|I|} \times T^{|J|})^{r_{\{I,J\}}} \to T^d \]
parameterising the locus in $X_r$ of tensors with a partition rank $\leq
r$ decomposition of a fixed type; this image is irreducible by virtue
of Lemma~\ref{lm:Pure} applied to the left-hand side above.
The total number of components of $X_r$
that we find is the number of ways of partitioning $r$ into $2^{d-1}-1$
nonnegative integers, which is polynomial in $r$. 

Given the (almost) linear relation between partition rank and analytic
rank \cite{Cohen21,Moshkovitz22}, it is natural to ask whether the number
of components of $Y_r$, too, is polynomial in $r$.
\end{ex}

\subsection{The vanishing ideal of a subset}

We will prove Noetherianity by looking at functions that
vanish identically on a subset.

\begin{de}
Given a subset $X \subseteq P$, we denote by $I_X(V) \subseteq K[P(V)]$
the ideal of all functions $P(V) \to K$ that vanish
identically on $X(V)$.  
\end{de}

We stress that conversely, since $K$ is finite, $X(V)$ is also the set of
all common zeros of $I_X(V)$ in $P(V)$. 

\subsection{Shifting and localising}

Definition~\ref{de:Shift} can be extended to subsets of
polynomial representations.

\begin{de}
Given a subset $X$ of a polynomial representation $P$ and a $U \in
\Vec$, the shift $\Sh_U X$ is the subset of $\Sh_U P$ defined by $(\Sh_U
X)(V):=X(U \oplus V)$.
\end{de}

\begin{de}
Given a subset $X$ of $P$ and a function $h \in K[P(0)]$, we can think of
$h$ as a function on any $P(V)$ via pullback along the linear map $P(V)
\to P(0)$, and hence also as a function on $X(V)$. We define $X[1/h]$
as the functor
\[ V \mapsto X[1/h](V):=\{p \in X(V) \mid h(p) \neq 0\}.
\qedhere \]
\end{de}

Clearly, $X[1/h]$ is a subset of $P$. 

We will often combine a shift and a localisation: given a function $h
\in K[P(U)]$, we can think of $h$ as a function on $K[(\Sh_U P)(0)]$,
and hence localise. 

\begin{ex} \label{ex:Matrices}
Let $P:V \to V \otimes V$ and let $X(V)$ be the set of
tensors (matrices) of rank $\leq n$ in $P(V)$. Let $h \in K[P(K^n)]$ be the $n \times n$-determinant
and set $U:=K^n$. Then $(\Sh_U X)[1/h]$ is isomorphic to the functor
that sends $V$ to $B \times V^{2n}$, where $B:=X(U)[1/h]$ is the set of
invertible $n \times n$-matrices, and the isomorphism $(\Sh_U X)[1/h](V)
\to B \times V^{2n}$ comes from observing that
\[ (\Sh_U X)[1/h](V) \subseteq P(U \oplus V)=(K^n \otimes
K^n) \times (K^n \otimes V) \times (V \otimes K^n) \times (V
\otimes V), \]
and realising that the $V \otimes V$-component of a matrix of rank
$\leq n$ is completely determined by its remaining three components,
provided that the $K^n \otimes K^n$-component has nonzero determinant.
This phenomenon, that $X$ becomes an affine space up to shifting and
localising, holds in greater generality, at least at a
counting level; see Corollary~\ref{cor:Shift}.
\end{ex}

\subsection{Reduction to the prime field case}

\begin{prop} \label{prop:PrimeField}
Suppose that Theorem~\ref{thm:Noetherianity} holds when $K$
is a prime field. Then it also holds when $K$ is an
arbitrary finite field.
\end{prop}

\begin{proof}
Let $F$ be the prime field of $K$ and set $e:=\dim_F K$. For
an $n$-dimensional $K$-vector space $U$, we write $U_F$ for the
$e \cdot n$-dimensional $F$-vector space obtained by restricting the scalar
multiplication on $U$ from $K \times U \to U$ to $F \times U \to U$.

Now let $P$ be a polynomial representation over $K$. Define a generic
representation $P_F$ over $F$ by setting, for a finite-dimensional
$F$-vector space $U$, $P_F(U):=(P(K \otimes_F U))_F$, and sending an
$F$-linear map $\phi:U \to V$ to the map $P_F(\phi)$, which is $K$-linear
and therefore also $F$-linear. It is easy to see from the definitions
that $P_F$ is polynomial of the same degree as $P$.

For a subset $X$ of $P$, we define a subset of $P_F$ via
$X_F(U):=X(K \otimes_F U)$. If $X_1 \supseteq X_2 \supseteq \ldots$ is a
chain of subsets in $P$, then $(X_1)_F \supseteq (X_2)_F \supseteq \ldots$
is a chain of subsets in $P_F$. By assumption, the latter
stabilises, say at $(X_{n_0})_F$. Then it follows that, for any
$n \geq n_0$ and any $m$, 
\[ X_{n}(K^m)=X_{n}(K \otimes_F F^m)=(X_{n})_F(F^m)
=(X_{n_0})_F(F^m)=X_{n_0}(K^m), \]
and this suffices to conclude that $X_n=X_{n_0}$. 
\end{proof}

\begin{center}
\fbox{\parbox{.8\textwidth}{
{\bf In view of Proposition~\ref{prop:PrimeField}, from
now on
we assume
that $K$ is a prime field.}}} 
\end{center}

An important reason for this assumption is that we can
then use Lemma~\ref{lm:Indep}. We believe that the proof below can be
adapted to arbitrary finite fields, and this might actually give more
general results. 
In particular, in the proof below we will act with the operators
$F_{n+1,j}=F_{n+1,j}[1]$; and in the general case we would have to work
with the operators $F_{n+1,j}[b]$ for $b \in \{1,\ldots,q-1\}$.  But the
reasoning below is already rather subtle, and we prefer not to make it
more opaque by the additional technicalities coming from non-prime fields.

\subsection{The embedding theorem} \label{ssec:Embedding}

We will prove Theorem~\ref{thm:Noetherianity} via an auxiliary result of
independent interest. Let $P$ be a polynomial representation of positive degree
$d$ and let $R$ an irreducible subobject of $P_{>d-1}$. Let $\pi:P
\to P/R=:P'$ be the projection. Dually, this gives rise to an embedding
$K[P/R] \subseteq K[P]$. For a fixed $V \in \Vec$, if we choose elements
$y_1,\ldots,y_n \in P(V)^*$ that map to a basis of $R(V)^*$, then we can
write elements of $K[P(V)]$ as reduced polynomials in $y_1,\ldots,y_n$
with coefficients that are elements of $K[P'(V)]$. We note, however,
that $R$ is typically not a direct summand of $P$. This implies, for
instance, that when acting with $\End(V)$ on $y_i$, we typically do not
stay within the linear span of the $y_1,\ldots,y_n$ but also get terms
that are linear functions in $K[P'(V)]$.

Let $X$ be a subset of $P$, and let $X'$ be the image of $X$ in $P/R$,
i.e. $X'(V):=\pi(X(V))$ (to simplify notation, we write $\pi$ instead
of $\pi_V$).

Now there are two possibilities: 
\begin{enumerate}
\item $X=\pi^{-1}(X')$, i.e., $X(V)=\pi^{-1}(X'(V))$ for all $V$. In
this case, $I_X$ is generated by $I_{X'} \subseteq K[P'] \subseteq K[P]$.

\item there exists a space $V$ and an element $f \in I_X(V)$ such that
$f$ does not lie in $K[P] \cdot I_{X'}$. 
\end{enumerate}

\begin{thm}[Embedding theorem]
Assume, as above, that $K$ is a prime field. From any $f \in I_X(V)
\setminus K[P(V)] \cdot I_{X'}(V)$, we can construct a $U \in \Vec$ and
a polynomial $h$ in $K[P(U)]$ of degree strictly smaller than that of
$f$, such that also $h$ does not vanish identically on $\pi^{-1}(X'(U))$
and such that the projection $\Sh_U P \to (\Sh_U P)/R$ restricts to
a injective map on $(\Sh_U X)[1/h]$.
\end{thm}

Here $(\Sh_U X)(V):=X(U \oplus V) \subseteq P(U \oplus V)=(\Sh_U P)(V)$
and $(\Sh_U X)[1/h]$ is the subset of $\Sh_U$ consisting of points $p$
where $h(p) \neq 0$. A warning here is that $h$ may actually vanish
identically on $X(U)$, in which case the conclusion is trivial because
$(\Sh_U X)[1/h]$ is empty. But in our application to the Noetherianity
theorem, this will be irrelevant.

\subsection{Proof of Noetherianity from the embedding theorem}

\begin{proof}
Proceeding by induction on $P$ along the partial order from
\S\ref{ssec:Order}, we may assume that Noetherianity holds for every
representation $Q \prec P$; we call this the {\em outer induction
assumption}. 

Let $d$ be the degree of $P$. If $d=0$, then $P(V)$ is a fixed finite set,
and clearly any chain of subsets stabilises. So we may assume that $d>0$.

Let $R$ be an irreducible sub-representation in the subrepresentation
$P_{>d-1}$ of $P$. Given a subset $X$ of $P$, we write $X'$ for its
projection in $P':=P/R$.

We define $\delta_X \in \{1,2,\ldots,\infty\}$ as the minimal degree of
a polynomial in $I_X \setminus K[P] \cdot I_{X'}$; this is
$\infty$ if $I_X = K[P] \cdot I_{X'}$. 

For $X,Y$ subsets of $P$ we write $X > Y$ if either $X' \supsetneq Y'$ or
else $X'=Y'$ but $\delta_{X'}>\delta_{Y'}$. Since, by the outer induction
assumption, $P'$ is Noetherian, this is a well-founded partial order on
subsets of $P$.
To prove that a given subset $X \subseteq P$ is Noetherian, we may
therefore assume that all subsets $Y \subseteq P$ with $Y<X$ are
Noetherian; this is the inner induction hypothesis.

Now if $\delta_X=\infty$, then any proper subset $Y$ of $X$ satisfies
$Y<X$, so we are done. We are therefore left with the case where $\delta_X
\in \ZZ_{\geq 1}$. 

Let $f \in I_X \setminus (K[P] \cdot I_{X'})$ be an element of degree
$\delta_X$. By the embedding theorem, there exists an element $h \in
K[P(U)] \setminus I_{X}(U)$ of degree $<\delta_X$ such that $(\Sh_U
X)[1/h] \to (\Sh_U P)/R$ is an injective map. Since $(\Sh_U P)/R \prec
P$, $(\Sh_U X)[1/h]$ is Noetherian by the outer induction hypothesis.

Define $Y$ as the subset of $X$ defined by the vanishing of
$h$. Explicitly,
\[ Y(V):=\{p \in X(V) \mid \forall \phi \in \Hom_{\Vec}(V,U):
h(P(\phi)p)=0\}. \]
Let $Y' \subseteq X'$ be the projection of $Y$ in $P/R$. If
$Y' \subsetneq X'$, then $Y<X$ and hence $Y$ is Noetherian
by the inner induction hypothesis. If $Y'=X'$, then $h \in
I_Y(U) \setminus (K[P(U)] \cdot I_{Y'}(U))$, and hence $\delta_Y \leq
\deg(h) < \delta_X$. So then, too, $Y<X$, and $Y$ is
Noetherian by the induction hypothesis. 

Now consider a chain 
\[ X \supseteq X_1 \supseteq X_2 \supseteq \ldots \]
of subsets. By the above two paragraphs, from some point on
both the chain $(X_i \cap Y)_i$ and the chain $((\Sh_U
X_i)[1/h])_i$ have stabilised. We claim that then also the chain
$(X_i)_i$ has stabilised. 

Indeed, take $p \in X_i(W)$. If $p \in X_i(W) \cap Y(W)$, then also $p \in
X_{i+1}(W) \cap Y(W)$ by the first chain, and we are done. If not, then
let $\phi:W \to U$ be a linear map such that $h(P(\phi)p) \neq
0$. Let $\iota:W \to U \oplus W$ be the embedding $w \mapsto
(\phi(w),w)$. Then we find that 
\begin{align*} &P(\iota)p \in X_i(U \oplus W)[1/h]
=(\Sh_U X_i)(W)[1/h]\\
&=(\Sh_U X_{i+1})(W)[1/h]
\subseteq X_{i+1}(U \oplus W). 
\end{align*}
Now if $\rho:U \oplus W \to W$ is the projection, then we
find that $p=P(\rho)P(\iota)p \in P(\rho)X_{i+1}(U \oplus
W)=X_{i+1}(W)$, as desired. 
\end{proof}

\subsection{Proof of the embedding theorem}

\begin{proof}
Recall that $P$ has degree $d>0$, $X \subseteq P$ is a subset,
$R$ an irreducible subrepresentaion of $P_{>d-1}$, $\pi:P \to P/R$
is the projection, $X'=\pi(X)$, $X \neq \pi^{-1}(X')$, and $f \in
I_X(V) \setminus (K[P(V)] \cdot I_{X'}(V))$. Assume that $f$ has
degree $\delta$. Recall from Lemma~\ref{lm:Dual} that $V^* \mapsto
K[P(V)]_{\leq \delta}$ is a polynomial representation. Furthermore,
this has subrepresentations $V^* \mapsto I_X(V)_{\leq \delta}$ and
$V^* \mapsto (K[P(V)] \cdot I_{X'}(V))_\leq \delta$. 

We may assume that $V=K^n$, and without loss of generality, $f$ is a
weight vector. We will act on $f$ with elements $g_{n+1,j}(s)^T$; see
Example~\ref{ex:Contravariant} for an explanation of the transpose.  The
part of degree $b$ in $s$ is then captured by the operator $F_{n+1,j}[b]$.

After acting repeatedly with operators $F_{n+1,j}[b]$ (for increasing
values of $n$ and possibly $j$ and observing that this does
not increase the degree of $f$), we may assume that the image of $f
\in K[P(K^n)]$ in the quotient representation 
\[ I_X(K^n)_{\leq \delta}
/(K[P(K^n)] \cdot I_{X'}(K^n))_{\leq \delta} \] 
is maximally spread
out (see Proposition~\ref{prop:MaxSpread}). After passing to a coordinate subspace,
by Remark~\ref{re:SpreadOut}, this implies that the weight of $f$
is $(1,\ldots,1)$. Moreover, it implies that if we split, for any
$j \in \{1,\ldots,n\}$, $\tilde{f}:=F_{n+1,j} f$ as $\tilde{f}_0 +
\tilde{f}_1$ where $\tilde{f}_0$ has weight $(1,\ldots,0,\ldots,1,1)$ and
$\tilde{f}_1$ has weight $(1,\ldots,q-1,\ldots,1,1)$, then $\tilde{f}_1$
vanishes identically on $X'(K^{n+1})$---indeed, otherwise $\tilde{f}_1$
would be a more spread-out polynomial that vanishes identically on $X$
but not on $X'$.

Choose a basis $\bx$ of $(P'(K^n))^* \subseteq P(K^n)^*$ consisting
of weight vectors, and extend this to a basis $\bx,\by$ of $P(K^n)^*$
of weight vectors. This means that $\by$ maps to a weight basis of
$R(K^n)^*$. Relative to these choices, we can write $f$ as a
reduced polynomial 
\begin{equation}
\label{eq:Decomp1}
f = \sum_\alpha f_\alpha(\bx) \by^\alpha 
\end{equation}
for suitable exponent vectors $\alpha$ and nonzero functions
$f_{\alpha} \in K[P'(K^n)]$. We choose this expression reduced relative to
$I_{X'}(K^n)$ in the following sense: no nonempty subset of the terms of
any $f_{\alpha}$ add up to a polynomial in $I_{X'}(K^n)$.
This implies that no $f_{\alpha}$ is in the ideal of $I_{X'}(K^n)$,
but the requirement is a bit stronger than that.

Let $y_0$ be one of the elements in $\by$ that appears in $f$; we further
choose $y_0$ such that the support in $\{1,\ldots,n\}$ of
its weight is inclusion-wise minimal. Consider the expression (coarser than
\eqref{eq:Decomp1}):
\[ f=f_0(\bx,\by\setminus \{y_0\}) y_0^0 + \dots +
f_{\ab}(\bx,\by \setminus \{y_0\}) y_0^{\ab} \]
where $f_e$ is a reduced polynomial in $\bx$ and the variables in $\by$ except
for $y_0$; and where $f_{\ab} \neq 0$ and ${\ab} \in \{1,\ldots,q-1\}$. Note
that $f_0$ is a weight vector of the same weight as $f$. {\em A priori},
the coefficients $f_e$ with $e>0$ need not be weight vectors, since the
weight monoid $(\{0,\ldots,q-1\}^n,\oplus)$ is not cancellative. However,
all terms in $f_e$ have the same weight up to identifying $0$ and
$q-1$, and upon adding $e$ times the weight of $y_0$
to any of them (using the operation $\oplus$), one obtains
the weight $(1,\ldots,1)$ of $f$.

\begin{lm}
We have ${\ab}=1$, $f_1$ is a weight vector, and after a permutation
$f_1$ has weight $(1^m,0^{n-m})$ and $y_0$ has weight
$(0^{m},1^{n-m})$ for some $m$.
\end{lm}

\begin{proof}
To prove the claim, let $j \in [n]$ be such that the weight
$\chi=(a_1,\ldots,a_n)$ of $y_0$ has $a_j>0$.

We partition the variables $\by$ into three subsets: those whose weight
has an entry $0$ in position $j$ are collected in the tuple $\by_0$;
those with a $1$ there in the tuple $\by_1$; and those with an entry $>1$
there in the tuple $\by_{>1}$.

We construct a weight basis of $P(K^{n+1})^*$ consisting of:
\begin{itemize}
\item $\bx,\by_0,\by_1,$ and $\by_{>1}$; 
\item the tuple $(n+1,j)\by_1$ obtained by applying
$(n+1,j)$ to each variable in $\by_1$; 
\item the tuple $F_{n+1,j}\by_{>1}$
obtained by applying $F_{n+1,j}$ to each variable in the
tuple $\by_{>1}$;
\item weight elements that together with $\bx$ form
a basis of $P(K^{n+1})^*$; and
\item weight elements that along with
$\by_0,\by_1,\by_{>1},(n+1,j)\by_1,F_{n+1,j}\by_{>1}$ project to a weight
basis of $R(K^{n+1})^*$.
\end{itemize}
The only non-obvious thing here is that the elements in
$F_{n+1,j}\by_{>1}$ can be chosen as part of a set mapping
to a basis of $R(K^{n+1})^*$, and this
follows from Lemma~\ref{lm:Indep}. Note that none of the variables
in the tuples $(n+1,j)\by_1$ and $F_{n+1,j}\by_{>1}$ has a $q-1$ on position $j$ of
its weight.

Either $y_0$ belongs to $\by_1$, or to $\by_{>1}$. In the first
case we define $y_1:=(n+1,j)y_0$, and in the second case we
define $y_1:=F_{n+1,j}y_0$. In both cases, $y_1$ is the (nonzero)
weight-$(a_1,\ldots,a_j-1,\ldots,a_n,1)$-component of $F_{n+1,j}y_0$
and one of the chosen variables.  (In the first case, this uses
Lemma~\ref{lm:Swap}.)

Consider
\begin{equation}  \label{eq:Action}
g_{n+1,j}(s)^T f = \sum_{e=0}^{\ab} (f_e(g_{n+1,j}(s)^T
\bx,g_{n+1,j}(s)^T (\by
\setminus \{y_0\})))(g_{n+1,j}(s)^T y_0)^e.
\end{equation}
From the ${\ab}$-th term we get a contribution $f_{\ab} \cdot {\ab} \cdot s \cdot
y_0^{{\ab}-1} \cdot y_1$, which is nonzero because ${\ab}<q-1$ and $q$ is prime.

Rewriting \eqref{eq:Action} as a reduced polynomial in $s, y_0,y_1$
with coefficients that are reduced polynomials in the remaining chosen
variables in $P(K^{n+1})^*$, we claim that $f_{\ab} \cdot {\ab}$ is precisely
the coefficient of $s \cdot y_0^{{\ab}-1} \cdot y_1$. Indeed, $y_0,y_1$
only appear in the terms $y_0=F_{n+1,j}[0] y_0$ and $F_{n+1,j}[1] y_0$
from $g_{n+1,j}(s)^T y_0$ and nowhere in $f_e(g_{n+1,j}(s)^T
\bx, g_{n+1,j}(s)^T (\by\setminus\{y_0\}))$ because:
\begin{itemize}
\item $g_{n+1,j}(s)^T$ maps the coordinates $\bx$ into linear combinations
of $\bx$ and the further chosen variables in $P'(K^{n+1})^*$;

\item $y_0,y_1$ do not appear in $F_{n+1,j}[b] \by$ for $b>1$ for weight
reasons: expressing the elements in the latter tuple on the basis of the
chosen variables, all variables have weights with a $b>1$
at position $n+1$, while $y_0,y_1$ have a $0$ and $1$ there, respectively;

\item $y_0$ does not appear in $F_{n+1,j}[1] y$ for any
variable $y$ in $\by$, again by comparing the weights in position $n+1$;

\item $y_1$ is different from all variables $F_{n+1,j}[1]y$ where $y$
ranges over the variables in $\by_{>1}$ (other than $y_0$, if $y_0$
is in $\by_{>1}$);

\item $y_1$ is different from all variables $(n+1,j)y$
where $y$ ranges over the variables in $\by_1$ (other than
$y_0$, if $y_0$ is in $\by_1$); and indeed

\item $y_1$ does not appear in the weight component $y'=F_{n+1,j}
y - (n+1,j) y$ of any variable $y$ in $\by_1$. Indeed, if
$(a_1',\ldots,1,\ldots,a_n')$ is the weight of $y$, then $y'$ has
weight $(a_1',\ldots,q-1,\ldots,a_n',1)$. But, as remarked earlier, the
variable $y_1$ constructed from $y_0$ does not have a $q-1$ on position
$j$ in its weight.
\end{itemize}

We conclude that, when writing $\tilde{f}=F_{n+1,j}f$ as a polynomial
in the chosen variables, the terms divisible by $y_0^{{\ab}-1} y_1$ are
precisely those in ${\ab} \cdot f_{\ab} \cdot y_0^{{\ab}-1} y_1$.  Now in $f_{\ab}$,
expanded as a reduced polynomial in $\by \setminus \{y_0\}$ with coefficients
that are reduced polynomials in $\bx$, consider any nonzero term $\ell(\bx)
(\by \setminus \{y_0\})^\alpha$. By reducedness of $f$, no nonempty
subset of the terms of $\ell(\bx)$ add up to a polynomial that vanishes
identically on $X'(K^n)$.  Group the terms in $\ell(\bx)$ into two
parts: $\ell(\bx)=\ell_0(\bx) + \ell_1(\bx)$, in such a manner that
$\ell_0(\bx)\cdot (\by \setminus \{y_0\})^\alpha y_0^{{\ab}-1} y_1$ is the
part of $\ell(\bx) (\by\setminus \{y_0\})^{\alpha} y_0^{{\ab}-1} y_1$ that has
weight $(1,\ldots,0,\ldots,1,1)$ and hence is part of $\tilde{f}_0$, and
$\ell_1(\bx)\cdot (\by\setminus \{y_0\})^{\alpha} y_0^{{\ab}-1}$ has weight
$(1,\ldots,q-1,\ldots,1,1)$ and hence is part of $\tilde{f}_1$. Since
$\tilde{f}_1$ vanishes identically on $\pi^{-1}(X')$, we find that
$\ell_1(\bx)$ does, too. Hence, since no nonempty set of
terms in $\ell$ adds up to a polynomial that vanishes on
$X'(K^n)$, we have  $\ell_1(\bx)=0$ and
$\ell(\bx)=\ell_0(\bx)$. It follows that $\ell(\bx)(\by
\setminus \{y_0\}^\alpha) y_0^{{\ab}-1} y_1$
is a weight vector of weight $(1,\ldots,0,\ldots,1,1)$.
Since the term $\ell(\bx)(\by \setminus \{y_0\})^\alpha$ in
$f_{\ab}(\bx,\by \setminus \{y_0\})$ was arbitrary, we find that
$f_{\ab} y_0^{{\ab}-1} y_1$ is a weight vector of weight
$(1,\ldots,0,\ldots,1,1)$. Since the weight of $y_0$ has a
positive entry on position $j$, we find that ${\ab}-1=0$ and all
weights appearing in $f_{\ab}$ have a $0$ on position $j$. 

Now $j$ was arbitrary in the support of the weight of $y_0$, so the
weights appearing in $f_{\ab}$ all have disjoint support from that of
$y_0$. But the only way, in the weight monoid
$\{0,1,\ldots,q-1\}^n$,
to obtain the weight $(1,\ldots,1)$ as a $\oplus$-sum of two weights
with disjoint supports is that, after a permutation, one weight is
$(1^m,0^{n-m})$ and the other weight is $(0^m,1^{n-m})$.  Hence $f_b$
is a weight vector that, after that permutation, has the former weight,
and then $y_0$ has the latter.
\end{proof}

Now we have found that
\[ f=f_0 + f_1 \cdot y_0   \]
where $f_1$ does not vanish identically on $\pi^{-1}(X')$; $f_1$ has weight
$(1^m,0^{n-m})$, $y_0$ has weight $\chi=(0^m,1^{n-m})$, and $f_0$ does
not involve $y_0$. It might be, though, that $f_0$ still
contains other variables $y$ in $\by$ of the same weight
$\chi=(0^{m},1^{n-m})$. Therefore, among the $\by$-variables, 
let $y_0=\hat{y}_1,\hat{y}_2,\ldots,\hat{y}_N$
be those that have weight equal to $\chi$; so $N$ is the multiplicity
of $\chi$ in $R(K^n)^*$. Then the above implies that
\begin{equation} \label{eq:finalf} 
f=\hat{f}_1 \hat{y}_1 + \cdots + \hat{f}_N \hat{y}_N + r 
\end{equation}
where each $\hat{f}_i$ has weight $(1^m,0^{n-m})$ and where the
$\by$-variables that appear in $r$ have weight vectors with at least
one nonzero entry in the first $m$ positions (here we use that $y_0$
had a weight vector of minimal support). Note that $\hat{f}_1=f_0$
does not vanish identically on $X'(K^n)$.

Now set $U:=K^{m}$, $W:=K^{n-m}$, and $h:=\hat{f}_1$.  Note that
$h \in K[P(U)]$, since its weight is $(1^m,0^{n-m})$. Also, $h$ has
lower degree than $f$, as desired, and does not vanish identically on
$\pi^{-1}(X'(U))$. In fact, all $\hat{f}_i$ are polynomials in $K[P(U)]$,
the $\hat{y}_i$ map to coordinates on $R(W)$, and $r$ is a polynomial
in $K[(\Sh_U P)(W)/R(W)]$ because every $\by$-variable in $r$ has at
least one nonzero entry among the first $m$ entries of its weight.

We claim that $(\Sh_U X)[1/h] \to (\Sh_U P)/R$ is injective.  We first
show that this is the case when evaluating at $W=K^{n-m}$. Consider two
points $p,p' \in (\Sh_U X)[1/h](W)$ with the same projection in
$(\Sh_U P)(W)/R(W)$, so that $p-p' \in R(W)$. Then $f$ vanishes at both $p$ and
$p'$ and, in \eqref{eq:finalf}, we have $\hat{f}_i(p)=\hat{f}_i(p')=:c_i
\in K$ for all $i$, as well as $r(p)=r(p')$. Then \eqref{eq:finalf}
shows that
\[ c_1 \hat{y}_1(p) + \cdots + c_N \hat{y}_N(p) =
c_1 \hat{y}_1(p') + \cdots + c_N \hat{y}_N(p').  \]
This can be expressed as $L(p-p')=0$ for a linear form $L \in R(W)^*$
which is nonzero because $c_1=h(p)=h(p') \neq 0$. Now act with an element
$\psi \in \End(K^{n-m})$ on \eqref{eq:finalf}, and then substitute $p$
and $p'$. This yields the identity $L(R(\psi)(p-p'))=0$. Hence we obtain
a nonzero $\End(K^{n-m})$-submodule of linear forms in $R(K^{n-m})^*$
that are zero on $p-p'$. But since $R$, and hence $R^*$, are irreducible
$\End(K^{n-m})$-modules by Remark~\ref{re:Irred}, this means that
$p-p'=0$. The same argument applies when $W$ is replaced by $K^{s}$
for any $s$. This completes the proof of the embedding theorem.
\end{proof}

\subsection{The weak shift theorem}

The embedding theorem can be used to show that the behaviour of
Example~\ref{ex:Matrices} is typical.

\begin{cor}[Weak shift theorem] \label{cor:Shift}
Suppose, as in the embedding theorem, that $K$ is a prime field of
cardinality $q$.  For any nonempty subset $X$ of some  polynomial 
generic representation $P$, there exist a $U \in \Vec$, a nonzero
function $h$ on $X(U)$, and a polynomial $A(n) \in \QQ[n]$ such that
for all $n \in \NN$ the cardinality of $(\Sh_U X)[1/h](K^n)$ equals
$q^{A(n)}$.
\end{cor}

In other words, at least in a counting sense, $(\Sh_U X)[1/h](K^n)$
is an affine space of dimension $A(n)$ over $K$. We expect there to
be a stronger version of this theorem, similar to the shift theorem in
\cite{Bik21}, which says that this affine space is functorial in $V$. But
we do not yet know the precise statement of this stronger theorem. Note
that, by applying the weak shift theorem to the subset $Y$ of $X$ defined
by the vanishing of $h$, and so on, we obtain a kind of stratification
of $X$ by finitely many affine spaces. A stronger version of the weak
shift theorem would therefore give deeper insight into the geometric
structure of general restriction-closed properties of tensors.

\begin{proof}[Proof of the weak shift theorem from the embedding theorem]
If $P$ has degree $0$, then $X=X(0)$ is a finite set, and we can choose
$U=0$ and $h$ to vanish on all but one point of $X(0)$, so that $(\Sh_U
X)[1/h]$ is that remaining point.

Now assume that $P$ has degree $d>0$ and that the result holds for
all polynomial representations $Q \prec P$. Let $R$ be an irreducible
subobject of $P_{>d-1}$ and let $X'$ be the projection of $X$ in
$P':=P/R$. 

There are two cases. First assume that $X$ is the preimage of $X'$.
Since $P' \prec P$, by the induction assumption there exists a $U$ and
an $h \in P'(U)$ that does not vanish on $X'(U)$ such that $|(\Sh_{U}
X')[1/h](K^n)|=q^{A(n)}$ for some polynomial $A(n)$.  Now $(\Sh_U
X)[1/h](K^n)$ is the preimage of $(\Sh_U X')[1/h](K^n)$, with fibres
$(\Sh_U R)(K^n)$. The fibre is a finite-dimensional vector space
over $K$ whose dimension is a polynomial $B(n)$. Hence $|(\Sh_{U}
X)[1/h'](K^n)|=q^{A(n)+B(n)}$, as desired.

If $X$ is not the preimage of $X'$, then we have seen that there exists
a $U_1 \in \Vec$, a polynomial $h_1 \in K[P(U_1)]$ that does not vanish on $X$,
and an injection
\[ (\Sh_{U_1} X)[1/h_1] \to ((\Sh_{U_1} P)/R)=:Q. \]
Let $Y$ be the image of this injection. Since $Q \prec P$, there
exist $U_2 \in \Vec$ and $h_2 \in K[Q(U_2)]$ such that $|(\Sh_{U_2}
Y)[1/h_2](K^n)|=q^{A(n)}$ for some polynomial $n$. Now set $U:=U_1 \oplus
U_2$ and $h:=h_1 \cdot h_2$ and we find that 
\[ |(\Sh_U X)[1/h](K^n)|=q^{A(n)}, \]
as desired. 
\end{proof}

\section{Relations to $\FI$ and algorithms} \label{sec:FI}

\subsection{$\FI$ and testing properties via subtensors}
\label{ssec:FI}

We recall from \cite{Church12} that $\FI$ is the category of finite sets
with injections and that an $\FI$-module over $K$ is a functor from $\FI$
to $\Vec$. The central result that we will use is the following.

\begin{thm}[\cite{Church12}] \label{thm:Church}
For any field $K$, any finitely generated $\FI$-module $M$ 
over $K$ is Noetherian in the sense that every
$\FI$-submodule is finitely generated. 
\end{thm}

We now use this to establish item (4) in
Corollary~\ref{cor:Main}.

\begin{thm}
Let $P$ be a polynomial generic representation over the
finite field $K$, and let $X \subseteq P$ be
a subset. Then there exists an $n_0$ such that for any $n \in \ZZ_{\geq
0}$, an element $p \in P(K^n)$ lies in $X(K^n)$ if and only if, for
every subset $S$ of $[n]$ of size $n_0$, the image of $p$ in $P(K^S)$
under the linear map corresponding to the coordinate projection $K^n
\to K^S$ lies in $X(K^S)$.
\end{thm}

\begin{proof}
Noetherianity for subsets of $P$ (Theorem~\ref{thm:Noetherianity})
implies that the ideal $I_X$ is finitely generated. In particular, $I_X$
is generated by $(I_X)_{\leq e}$ for some finite degree $e$.  Now consider
the functor $F$ from $\FI$ to $\Vec$ that assigns to any finite set $S$
the space $K[P(K^S)]_{\leq e}$ and to every injection $\iota: S \to T$
the embedding $K[P(K^S)]_{\leq e} \to K[P(K^T)]_{\leq e}$ coming from
the pullback along the linear map $P(K^T) \to P(K^S)$ associated to
$\iota$. Since weights in $K[P(K^S)]_{\leq e}$ have at most $de$ nonzero
entries, where $d=\deg(P)$ (see Lemma~\ref{lm:Dual}), $F$ is generated
by $F([de])$, hence a finitely generated $\FI$-module. By
Theorem~\ref{thm:Church}, 
the $\FI$-submodule $S \mapsto I_X(K^S)$ is also finitely generated,
say by $I_X(K^{n_0})$. This $n_0$ has the desired property.
\end{proof}

\subsection{Algorithms}

As noted in the introduction, the theorem just established implies the
existence of a polynomial-time algorithm for testing the
property $X$.

\subsection{Infinite tensors}

We conclude with a theorem about infinite tensors. 
Let $P:\Vec \to \Vec$ be a polynomial generic representation
over the finite field $K$. Define 
\[ P_{\infty}:=\lim_{\ot n} P(K^n) \]
where the limit is along the projections $P(K^{n+1}) \to P(K^n)$ coming
from the projections $K^{n+1} \to K^n$ forgetting the last entry. In the
case where $P(V)=V^{\otimes d}$, $P_\infty$ can be thought of as the space
of $\NN \times \cdots \times \NN$-tensors (with $d$ factors $\NN$). The
space $P_\infty$ carries the inverse limit of discrete topologies, or,
equivalently, the Zariski topology in which closed subsets are defined
by the vanishing of (possibly infinitely many) functions in the ring
\[ K[P_\infty]:=\lim_{n \to \infty} K[P(K^n)]. \]
On $P_\infty$ and $K[P_\infty]$ acts the monoid $\Pi$ of matrices
that differ from the identity matrix only in finitely many
positions. Let $X \subseteq P$ be a subset in the sense of
Definition~\ref{de:Subset}. Then $X_\infty:=\lim_{\ot n} X(K^n)$ is a
subset of $P_\infty$. 

\begin{prop}
The correspondence that sends $X \subseteq P$ to $X_\infty \subseteq
P_\infty$ is a bijection between subsets of $P$ and closed, $\Pi$-stable
subsets of $P_\infty$.
\end{prop}

\begin{proof}
The subset $X_\infty$ is clearly closed and $\Pi$-stable.  Conversely,
let $Y \subseteq P_\infty$ be closed and $\Pi$-stable.  Define $Y_n
\subseteq P(K^n)$ as the image of $Y$ under the projection $P_\infty \to
P_n$, and for any $V \in \Vec$ define $X(V)$ to be the image of $Y_n$
under $P(\phi)$ for any linear isomorphism $K^n \to V$; this
is independent of the choice of $\phi$.

We claim that $X$ is a subset of $P$. Indeed, if $\psi:V \to W$
is any linear map and $\phi:K^n \to V$ a bijection, then we have to
show that $P(\psi)P(\phi)Y_n$ is contained in $P(\phi')Y_{m}$ where
$\phi':K^{n'} \to W$ is a bijection.  Now the map $(\phi')^{-1} \psi
\phi: K^n \to K^{n'}$ extends to a linear map $\alpha:K^m \to K^m$, where
$m:=\max\{n,n'\}$ and we regard $K^n$ and $K^{n'}$ as the subspaces of
$K^m$ where the last $m-n$ and $m-n'$ entries, respectively, are zero. The
map $\alpha$, in turn, can be regarded an element of $\Pi$. Consequently,
$Y$ is invariant under $\alpha$, and therefore so is $Y_m$. This shows
that the map $P((\phi')^{-1} \psi \phi)$ maps $Y_n$ into $Y_{n'}$, so that
$P(\psi)$ maps $X(V)$ into $X(W)$, as desired.

Next we claim that $X_\infty=Y$. Indeed, the fact that $Y$ is closed
means precisely that to test whether $(y_0,y_1,y_2,\ldots)$ lies in $Y$,
it suffices to check whether $y_n$ lies in $Y_n$ for all $y$. And on
the other hand, this is precisely the definition of $X_\infty$.
\end{proof}

Theorem~\ref{thm:Noetherianity} now implies the following.

\begin{cor}
Closed, $\Pi$-stable subsets of $P_\infty$ satisfy the descending chain
condition. Dually, $\Pi$-stable ideals of $K[P_\infty]$ satisfy the
ascending chain condition. \hfill $\square$
\end{cor}

The latter, ring-theoretic Noetherianity is not known in the context
of infinite fields \cite{Draisma17}, except in characteristic zero for
a few degree-two polynomial functors \cite{Nagpal15,Sam22}. However,
in the current context, all ideals are radical, since in $K[P_\infty]$
every element $f$ satisfies the identity $f^q=f$. This implies that the
two statements in the corollary are equivalent.

\bibliographystyle{alpha}
\bibliography{draismajournal,diffeq}

\end{document}